\documentclass{amsart}
\usepackage{amsmath, amsthm, amsfonts, amssymb, amscd}
\usepackage[active]{srcltx}
\usepackage[pdftex]{graphicx}
\usepackage{color} 

\newcommand{\La}{\Lambda}
\newcommand{\la}{\lambda}
\newcommand{\be}{\beta}
\newcommand{\e}{\epsilon}
\newcommand{\de}{\delta}
\newcommand{\ga}{\gamma}

\newcommand{\Si}{\Sigma}
\newcommand{\s}{\sigma}
\newcommand{\al}{\alpha}
\newcommand{\cal}{\mathcal}

\newcommand{\aC}{\mathbf{C}}
\newcommand{\dist}{\operatorname{dist}}
\newcommand{\diam}{\operatorname{diam}}

\newtheorem{theorem}{Theorem}[section]
\newtheorem{corollary}[theorem]{Corollary}
\newtheorem{proposition}[theorem]{Proposition}
\newtheorem{lemma}[theorem]{Lemma}
\theoremstyle{definition}

\newtheorem{remark}[theorem]{Remark}
\newtheorem{claim}{Claim}
\newtheorem{definition}[theorem]{Definition}
\newtheorem{notation}[theorem]{Notation}

\newcommand{\CR}{\mathcal{CR}}
\newcommand{\ie}{i.e.\ }
\newcommand{\R}{\mathbb{R}}
\newcommand{\Z}{\mathbb{Z}}
\newcommand{\N}{\mathbb{N}}
\newcommand{\A}{\mathbb{A}}
\newcommand{\T}{\mathbb{T}}
\newcommand{\inter}{\operatorname{int}}
\newcommand{\til}{\tilde}
\newcommand{\sm}{\setminus}
\newcommand{\abs}[1]{\left\lvert{#1}\right\rvert}
\newcommand{\norm}[1]{\left\|{#1}\right\|}
\newcommand{\Per}{\operatorname{Per}}
\newcommand{\Fix}{\operatorname{Fix}}
\newcommand{\SL}{\operatorname{SL}}
\newcommand{\tr}{\operatorname{tr}}
\newcommand{\bd}{\partial}
\newcommand{\ol}{\overline}

\newcommand{\id}{\operatorname{id}}

\title[Some Consequences of the Shadowing Property]{Some Consequences of the Shadowing Property\\ in Low Dimensions}
\author{Andres Koropecki}
\address{Andres Koropecki. Universidade Federal Fluminense, Instituto de Matem\'atica, Rua M\'ario Santos Braga S/N, 24020-140 Niteroi, RJ, Brasil}
\email{ak@id.uff.br}
\author{Enrique R. Pujals}
\address{Enrique R. Pujals. IMPA, Estrada Dona Castorina 110, Rio de Janeiro, 22460-320, RJ, Brasil}
\email{enrique@impa.br}
\thanks{The authors were partially supported by CNPq-Brasil}

\begin{document}

\begin{abstract}
We consider low-dimensional systems with the shadowing property and we study the problem of existence of periodic orbits. In dimension two, we show that the shadowing property for a homeomorphism implies the existence of periodic orbits in every $\epsilon$-transitive class, and in contrast we provide an example of a $C^\infty$ Kupka-Smale diffeomorphism with the shadowing property exhibiting an aperiodic transitive class.  Finally we consider the case of transitive endomorphisms of the circle, and we prove that the $\alpha$-H\"older shadowing property with $\alpha>1/2$ implies  that the system is conjugate to an expanding map.
\end{abstract}
\maketitle

\section{Introduction}
The main goal of this article is to obtain dynamical consequences of
the shadowing property for surface maps and  one-dimensional dynamics.

Let $(X,d)$ be a metric space and $f\colon X\to X$ a homeomorphism. A
(complete) $\delta$-pseudo-orbit for $f$ is a sequence $\{x_n\}_{n\in
\Z}$ such that $d(f(x_n), x_{n+1})<  \delta$ for all $n\in \Z$. We
say that the orbit of $x$ $\epsilon$-shadows the given pseudo-orbit if
$d(f^n(x), x_n)<\epsilon$ for all $n\in \Z$. Finally, we say that $f$
has the shadowing property (or pseudo-orbit tracing property) if for
each $\epsilon>0$ there is $\delta>0$ such that every
$\delta$-pseudo-orbit is $\epsilon$-shadowed by an orbit of $f$. Note
that we do not assume uniqueness of the shadowing orbit.

One motivation to study systems with this property is that numeric simulations of dynamical systems always produce pseudo-orbits. Thus, systems with the shadowing property are precisely the ones in which numerical simulation does not introduce unexpected behavior, in the sense that simulated orbits actually ``follow'' real orbits.

 When one considers pseudo-orbits,
the natural set that concentrates the nontrivial dynamics is the chain  
recurrent set. If
the system has the shadowing property, the closure of the recurrent  
set coincides with
the chain recurrent set. As usual, it is then natural to ask about
the existence of periodic orbits in the recurrent set. If one also assumes that the system is expansive in the recurrent set (as it happens for hyperbolic systems), then it is easy to see that periodic orbits are dense in the recurrent set. But is there a similar property without requiring expansivity? Without any additional assumption on the underlying space, the answer is no: there are examples on a totally disconnected metric space which have the shadowing property and have no periodic points (moreover, they are minimal); for instance, the so-called adding-machines \cite{adding}. 

However, if the space is assumed to be connected, and in particular a connected manifold, it is not even known wether the shadowing property implies the existence of at least one periodic orbit.
Our first result addresses this problem in dimension two (for precise definitions see \S\ref{sec:etrans}).

\begin{theorem}\label{th:A} Let $S$ be a compact orientable surface,
and let $f\colon S\to S$ be a homeomorphism with the shadowing
property. Then for any given $\epsilon>0$, each $\epsilon$-transitive
component has a periodic point.
\end{theorem}

As an immediate consequence, we have
\begin{corollary} If a homeomorphism of a compact surface has the shadowing property, then it has a periodic point.
\end{corollary}

Note that Theorem \ref{th:A} does not rule out the existence of aperiodic chain
transitive components (in fact Theorem \ref{th:B} below provides an example with the shadowing property exhibiting an aperiodic chain transitive class). However, if there is such a component, some of its points must be accumulated by periodic points.

\begin{corollary} Let $S$ be a compact orientable surface, and let
$f\colon S\to S$ a homeomorphism with the shadowing property. Then
$\ol{\Per(f)}$ intersects every chain transitive class.
\end{corollary}

Thus, the presence of an aperiodic chain transitive class implies that there are infinitely many periodic points. In particular, we obtain the following:

\begin{theorem} \label{th:MS} Let $S$ be a compact orientable surface, and let $f\colon S\to S$ be a Kupka-Smale diffeomorphism with the shadowing property. If there are only finitely many periodic points, then $f$ is Morse-Smale.
\end{theorem}

Another problem of interest is to find ``new''  examples of systems having the shadowing property.  It is known that in dimension at least $2$, topologically stable systems have the shadowing property \cite{walters, nitecki}.  Systems which are hyperbolic
(meaning Axiom A with the strong transversality condition) also
exhibit the shadowing property. In fact, for such systems, a stronger
property called Lipschitz shadowing holds. This means that there is
some constant $C$ such that in the definition of shadowing one can
always choose $\delta=C\epsilon$ \cite[\S2.2]{pilyugin}. In
\cite{tikhomirov} it is shown that Lipschitz shadowing for
diffeomorphisms is actually equivalent to hyperbolicity.

However, not all systems with the shadowing property have Lipschitz
shadowing. In fact, in view of \cite{tikhomirov}, a simple example is
any non-hyperbolic system which is topologically conjugated to a
hyperbolic one.  Nevertheless, this type of example still has many of the
properties of hyperbolicity; in particular, there are finitely many chain transitive classes, with dense periodic orbits.
Another non-hyperbolic example which has the shadowing property is a
circle homeomorphism with infinitely many fixed points, which are
alternatively attracting and repelling and accumulating on a unique
non-hyperbolic fixed point. However, in this type of example, although there are infinitely many chain transitive classes, they all have dense periodic points (in fact they are periodic points).

The next theorem gives a new type of example of smooth diffeomorphism with the shadowing property, 
which is essentially different from the other known examples in that it has an aperiodic chain transitive class, and moreover, all periodic points are hyperbolic. In particular, our example shows that one cannot hope to improve Theorem \ref{th:A} by ``going to the limit'', even for Kupka-Smale diffeomorphisms:

\begin{theorem}\label{th:B} In any compact surface $S$, there exists a
Kupka-Smale $C^\infty$ diffeomorphism $f\colon S\to S$ with the
shadowing property which has an aperiodic chain transitive component.
More precisely, it has a component which is an invariant circle supporting an irrational
rotation.
\end{theorem}

At this point, we want to emphasize that none of the theorems stated so far assume either Lipschitz or H\"older shadowing.

Finally, we consider the case of transitive endomorphisms of the circle with the $\al$-H\"older shadowing property, \ie such that there is a constant $C$ such that every $\delta$-pseudo-orbit is $C\delta^\alpha$-shadowed by an orbit, and we show that $\al$-H\"older shadowing with $\al>1/2$ implies conjugacy to linear expanding maps (see definitions \ref{holder}-\ref{expanding} for details).

\begin{theorem}\label{crit-1} Let $f$ be a $C^{2}$ endomorphism of the circle with finitely many turning points. Suppose that $f$ is  transitive and satisfies the $\al$-H\"older shadowing property with $\al>\frac{1}{2}$. Then $f$ is conjugate to a linear expanding endomorphism.
\end{theorem}

If the transitivity of $f$ persists after perturbations, we can improve the result. We say that $f$ is $C^r$ robustly transitive if all maps in a $C^r$-neighborhood of $f$ are transitive.

\begin{theorem} \label{crit-robust} Let $f$ be a $C^{2}$ orientation preserving endomorphism of the circle with finitely many turning points. Suppose that $f$ satisfies the $\al$-H\"older shadowing property with $\al>\frac{1}{2}$. If $f$ is $C^r$-robustly transitive, $r\geq 1$, then $f$ is an expanding endomorphism. 
\end{theorem}

Observe that in this theorem we do not assume that the shadowing property holds for perturbations of the initial system. 
Theorem \ref{crit-robust} can be concluded directly from \cite{KSS}, where it is proved that hyperbolic endomorphisms are open and dense for one dimensional dynamics. Nevertheless, we provide the present proof because it involves very elementary ideas that may have a chance to be generalized for surface maps.   
A result similar to Theorem \ref{crit-1} for the case of diffeomorphisms in any dimension was recently given in \cite{tikhomirov-holder}.

Let us say a few words about the techniques used in this article.
To obtain Theorem \ref{th:A}, we use Conley's theory combined with a  
Lefschetz index
argument to reduce the problem to one in the annulus or the torus. To do this, we prove a result about aperiodic $\epsilon$-transitive components that is unrelated to the shadowing property and may be interesting by itself (see Theorem \ref{th:annuli}). In  
that setting, then we apply Brouwer's theory for plane homeomorphisms (which is strictly two-dimensional) to obtain the required periodic points.

To prove Theorem \ref{th:B}, we use a construction in the annulus with  
a special kind of
hyperbolic sets, called crooked horseshoes, accumulating on an  
irrational rotation on the
circle (with Liouville rotation number). The shadowing property is  
obtained for points
far from the rotation due to the hyperbolicity of the system outside a  
neighborhood of
the rotation, and near the rotation the shadowing comes from the  
crooked horseshoes. The
main technical difficulty for this construction is to obtain  
arbitrarily close to the
identity a hyperbolic system with a power exhibiting a crooked  
horseshoe. This is
addressed in the Appendix (see Proposition \ref{pro:flow})

To prove Theorem \ref{crit-1} we use the shadowing property to obtain a small interval containing a turning point, such that some forward iterate of the interval also contains a turning point. Assuming that the shadowing is H\"older with $\alpha>\frac{1}{2}$, it is concluded that some forward iterate of the interval has to be contained inside the initial interval, contradicting the transitivity. Once turning points are discarded, Theorem \ref{crit-robust} is proved using the fact that the dynamics preserves orientation and recurrent points can be closed to periodic points by composing with a translation. 

\subsection*{Acknowledgements}
We thank the anonymous referees for their very careful reading of this paper and for the valuable comments and suggestions.

\section{Shadowing implies periodic orbits: proof of Theorem \ref{th:A}}

\subsection{Lifting pseudo-orbits}
Let $S$ be a orientable surface of finite type, and $f\colon S\to S$ a homeomorphism. If $f$ is not the sphere, we may assume that $S$ is endowed with a complete Riemannian metric of constant non positive curvature, which induces a metric $d(\cdot, \cdot)$ on $S$. Denote by $\hat{S}$ the universal covering of $S$ with covering projection $\pi\colon \hat{S}\to S$, equipped with the lifted metric which we still denote by $d(\cdot,\cdot)$ (note that  $\hat{S}\simeq \R^2$ or $\mathbb{H}^2$).

The covering projection $\pi$ is a local isometry, so that we may fix $\epsilon_0$ such that for each $\hat{x}\in \hat{S}$ there is $\epsilon_0>0$ such that $\pi$ maps the $\epsilon_0$-neighborhood of $\hat{x}$ to the $\epsilon_0$-neighborhood of $\pi(\hat{x})$  isometrically for any $\hat{x}\in \hat{S}$.

The next proposition ensures that one can always lift $\epsilon$-pseudo orbits of $f$ to the universal covering in a unique way (given a base point in a fiber) if $\epsilon$ is small enough.

\begin{proposition}\label{pro:lift}  Given a lift $\hat{f}\colon \hat{S}\to \hat{S}$ of $f$, an $\epsilon_0$-pseudo orbit $\{x_i\}$, and $\hat{y}\in \pi^{-1}(x_0)$, there is a unique $\epsilon_0$-pseudo orbit $\{\hat{x_i}\}$ for $\hat{f}$ such that $\hat{x}_0 =\hat{y}$ and $x_i = \pi(\hat{x_i})$ for all $i$.
\end{proposition}
\begin{proof} Note that from the definition of $\epsilon_0$, we have $$\epsilon_0 < \min\{d(\hat{y},\hat{x}): x\in S,\, \hat{x},\hat{y}\in \pi^{-1}(x),\, \hat{y}\neq\hat{x}\}.$$
 Set $\hat{x}_0 = \hat{y}$. Then there is a unique choice of $\hat{x}_1\in \pi^{-1}(x_1)$ such that $d(\hat{f}(\hat{x}_0), \hat{x}_1)<\epsilon_0$, and similarly there is a unique $\hat{x}_{-1}\in \pi^{-1}(x_{-1})$ such that $d(\hat{f}(\hat{x}_{-1}), \hat{x}_0)<\epsilon_0$. Proceeding inductively, one completes the proof.
\end{proof}

The following proposition follows from a standard compactness argument which we omit.

\begin{proposition}\label{pro:uniform-lift} If $K\subset S$ is compact and $\hat{f}\colon \hat{S}\to \hat{S}$ is a lift of $f$, then $\hat{f}$ is uniformly continuous on the $\epsilon$-neighborhood of $\pi^{-1}(K)$, for some $\epsilon>0$.
\end{proposition}

We say that $f$ has the shadowing property in some invariant set $K$ if for every $\epsilon>0$ there is $\delta>0$ such that every $\delta$-pseudo orbit in $K$ is $\epsilon$-shadowed by some orbit (not necessarily in $K$).

\begin{proposition} \label{pro:shadow-lift} Suppose that $f$ has the shadowing property in a compact set $K\subset S$. If $\hat{f}\colon \hat{S}\to \hat{S}$ is a lift of $f$, then $\hat{f}$ has the shadowing property in $\pi^{-1}(K)$.
\end{proposition}

\begin{proof} From Proposition \ref{pro:uniform-lift}, given $\epsilon>0$ we may choose $\epsilon'<\min\{\epsilon, \epsilon_0/3\}$ such that $\hat{d}(\hat{f}(x), \hat{f}(y))<\epsilon_0/3$ whenever $d(x,y)<\epsilon'$, and a similar condition for $\hat{f}^{-1}$.

Let $\delta<\epsilon_0/3$ and let $\{\hat{x}_n\}$ be a $\delta$-pseudo orbit of $\hat{f}$ in $\pi^{-1}(K)$. Then $\{\pi(\hat{x}_n)\}$ is a $\delta$-pseudo orbit of $f$ in $K$. If $\delta$ is small enough, then $\{\pi(\hat{x}_n)\}$ is $\epsilon'$-shadowed by the orbit of some point $x\in M$. Since $d(x, \pi(\hat{x}_0))<\epsilon'$, if $\hat{x}$ is the element of $\pi^{-1}(x)$ closest to $\hat{x}_0$ we have $d(\hat{x},\hat{x}_0)<\epsilon'$. We know that $d(\hat{f}(\hat{x}_0),\hat{x}_1)<\delta$, and from our choice of $\epsilon'$ also $d(\hat{f}(\hat{x}), \hat{f}(\hat{x}_0))< \epsilon_0/3$,  so $d(\hat{f}(\hat{x}), \hat{x}_1)<\delta+\epsilon_0/3$. On the other hand, since $d(\pi(\hat{f}(\hat{x})), \pi(\hat{x}_1))<\epsilon'$, we must have that $d(\hat{f}(\hat{x}), T\hat{x}_1)<\epsilon'<\epsilon_0/3$ for some covering transformation $T$. But then $d(T\hat{x}_1, \hat{x}_1)<\delta+2\epsilon_0/3 <\epsilon_0$. This implies that $T=Id$ so that $d(\hat{f}(\hat{x}),\hat{x}_1)<\epsilon'$. In particular $\hat{f}(\hat{x})$ is the element of $\pi^{-1}(f(x))$ closest to $\hat{x}_1$, so we may repeat the previous argument inductively to conclude that  $d(\hat{f}^n(\hat{x}), \hat{x}_n)<\epsilon'$ for all $n\geq 0$. 

If $\hat{y}$ is the element of $\pi^{-1}(f^{-1}(x))$ closest to $\hat{x}_{-1}$, then $d(\hat{y}, \hat{x}_{-1})<\epsilon'$. By the previous argument starting from $\hat{x}_{-1}$ instead of $\hat{x}_0$, we have that $d(\hat{f}(\hat{y}), \hat{x}_0)<\epsilon'<\epsilon_0/3$, and this means that $\hat{f}(\hat{y})$ is the element of $\pi^{-1}(x)$ closest to $\hat{x}_0$ (which we named $\hat{x}$ before). Thus $\hat{y}=\hat{f}^{-1}(\hat{x})$, and we conclude that  $d(\hat{f}^{-1}(\hat{x}), \hat{x}_{-1})<\epsilon'$.
By an induction argument again,  we conclude that $d(\hat{f}^n(\hat{x}), \hat{x}_n)<\epsilon'$ for $n<0$ as well. This completes the proof.
\end{proof}

\subsection{Shadowing and periodic points for surfaces}

First we state the following well-known consequence of Brouwer's plane translation theorem (see, for instance, \cite{franks-brouwer}).
\begin{theorem}[Brouwer] Let $f\colon \R^2\to \R^2$ be an orientation preserving homeomorphism. If $f$ has a nonwandering point, then $f$ has a fixed point. 
\end{theorem}

Suppose $f\colon S \to S$ is homotopic to the identity, and let $\hat{f}\colon \hat{S}\to \hat{S}$ be the lift of $f$ obtained by lifting the homotopy from the identity. Then it is easy to see that $\hat{f}$ commutes with covering transformations.

\begin{theorem}\label{th:open-periodic} Let $f\colon S \to S$ be a homeomorphism homotopic to the identity. Suppose that there is a compact invariant set $\Lambda$ where $f$ has the shadowing property. Then $f$ has a periodic point. 
\end{theorem}

\begin{proof} We will also assume that the metric in $S$ is as in the previous section, which we may since the shadowing property in a compact set is independent of the choice of Riemannian metric on $S$.

We may assume that $S$ is not the sphere, since in that case $f$ would have a periodic point by the Lefschetz-Hopf theorem.
Consider the lift $\hat{f}\colon  \hat{S}\to \hat{S}$ of $f$ obtained by lifting the homotopy from the identity (which therefore commutes with the covering transformations).

By Proposition \ref{pro:shadow-lift}, $\hat{f}$ has the shadowing property in $\pi^{-1}(\Lambda)$. Fix $\epsilon>0$, and let $\delta>0$ be such that every $\delta$-pseudo orbit in $\pi^{-1}(K)$ is $\epsilon$-shadowed by an orbit of $\hat{f}$. Since $\Lambda$ is compact and invariant, there is a recurrent point $x\in \Lambda$, and so if $\hat{x}\in \pi^{-1}(x)$ we can find $n>0$ and a covering transformation $T$ such that $d(\hat{f}^n(\hat{x}), T\hat{x})<\delta$. Since $T$ is an isometry and commutes with $\hat{f}$, the sequence
$$\dots, T^{-1}\hat{f}^{n-1}(\hat{x}), \hat{x}, \hat{f}(\hat{x}), \dots, \hat{f}^{n-1}(\hat{x}), T\hat{x}, T\hat{f}(\hat{x}), \dots, T\hat{f}^{n-1}(\hat{x}), T^2\hat{x},\dots$$ 
is a $\delta$-pseudo orbit, and so it is $\epsilon$-shadowed by the orbit of some $\hat{y}\in \hat{S}$. This implies in particular that $d(\hat{f}^{kn}(\hat{y}), T^k\hat{x})<\epsilon$ for all $k\in \Z$, so that $d((T^{-1}\hat{f}^n)^k(\hat{y}), \hat{y})<\epsilon$ for all $k\in \Z$. Note that $T^{-1}\hat{f}^n$ is a homeomorphism of $\hat{S}\simeq \R^2$, and we may assume that it preserves orientation without loss of generality. Moreover, we have from the previous facts that the closure of the orbit of $\hat{y}$ for $T^{-1}\hat{f}^n$ is a compact invariant set; thus $T^{-1}\hat{f}^n$ has a recurrent point, and by Brouwer's Theorem, it has a fixed point. Since $T^{-1}\hat{f}^n$ is a lift of $f^n$, we conclude that $f$ has a periodic point. This completes the proof.
\end{proof}

\begin{corollary}\label{coro:annulus} If $f\colon \A := \mathbb{S}^1\times \R\to \A$ is a homeomorphism of the open annulus, and $f$ has the shadowing property on some compact set $\Lambda$, then $f$ has a periodic point.
\end{corollary}
\begin{proof} 
Note that $f^2$ is homotopic to the identity (this is an easy consequence of the fact that $f$ is a homeomorphism and $\mathbb{S}^1\times \{0\} \simeq \mathbb{S}^1$ is a deformation retract of $\A$; see for instance \cite{epstein}). Thus, noting that $f^2$ still has the shadowing property in $\Lambda$, the existence of a periodc point follows from Theorem \ref{th:open-periodic} applied to $f^2$.
\end{proof}

\begin{corollary}\label{coro:torus} Let $f\colon \T^2\to \T^2$ be a homeomorphism with the shadowing property. Then $f$ has a periodic point.
\end{corollary}
\begin{proof}
Suppose that $f$ has no periodic points. Using $f^2$ instead of $f$ we may assume that $f$ preserves orientation. By the Lefschetz-Hopf theorem, we only have to consider the cases where $f$ is homotopic to the identity or to a map conjugated to a power of the Dehn twist $D\colon (x,y)\mapsto(x+y, y)$ (otherwise, $f$ has a periodic point even without assuming the shadowing property). 

If $f$ is isotopic to the identity, $f$ has a periodic point by Theorem \ref{th:open-periodic}. 
Now suppose that $f$ is homotopic to a map conjugated to $D^m$ for some $M\in \Z$. Using a homeomorphism conjugated to $f$ instead of $f$, we may assume that $f$ is in fact is homotopic to $D^m:(x,y)\mapsto (x+my, y)$. Let $\tau\colon \A\to \T^2$ be the covering map $(x,y)\mapsto (x+\Z,y)$, where $\A=\mathbb{S}^1\times \R$. Since $f$ is isotopic to $D^m$, we can lift $f$ by $\tau$ to a homeomorphism $\til{f}\colon \A\to \A$, which is homotopic to the identity.

Note that in the proof of Propositions \ref{pro:lift} and \ref{pro:shadow-lift} we did not use the fact that $\pi$ was the universal covering map. Thus by the same argument applied to the covering $\tau$ one sees that $\til{f}$ has the shadowing property in $\A = \tau^{-1}(\T^2)$. Moreover, following the proof of Theorem \ref{th:open-periodic}, we see that there is a point $\til{z}\in \A$ and a covering transformation $T\colon \A\to \A$ such that $d(\til{f}^{kn}(\til{z}), T\til{z})<\epsilon$ for all $k\in \Z$. But then, noting that T commutes with $\til{f}$, we see that $T^{-1}\til{f}^{n}$ is a lift of $f^n$ which has a compact invariant set $\Lambda$ where the shadowing property holds (namely, the closure of the orbit of $\til{z}$). The previous corollary applied to $\til{f}$ implies that $\til{f}$ (and thus $f$) has a periodic point.

\end{proof}

\subsection{Lyapunov functions and $\epsilon$-transitive components}
\label{sec:etrans}

Let $f\colon X\to X$ be a homeomorphism of a compact metric space $X$.
Denote by $\CR(f)$ the chain recurrent set of $f$, \ie $x\in \CR(f)$ if for every $\epsilon>0$ there is an $\epsilon$-pseudo-orbit for $f$ connecting $x$ to itself.

The chain recurrent set is partitioned into chain transitive classes, defined by the equivalence relation $x\sim y$ if for every $\epsilon>0$ there is an $\epsilon$-pseudo-orbit from $x$ to $y$ and another from $y$ to $x$.
The chain transitive classes are compact invariant sets.

Recall that a complete Lyapunov function for $f$ is a continuous function $g\colon X\to \R$ such that
\begin{enumerate}
\item $g(f(x))<g(x)$ if $x\notin \CR(f)$;
\item If $x,y\in \CR(f)$, then $g(x) = g(y)$ if and only if $x$ and $y$ are in the same chain transitive component;
\item $g(\CR(f))$ is a compact nowhere dense subset of $\R$.
\end{enumerate}

Given a Lyapunov function as above, we say that $t\in \R$ is a \emph{regular value} if $g^{-1}(t)\cap \CR(f)=\emptyset$. Note that the set of regular values is open and dense in $\R$.

We recall the following result from Conley's theory (see \cite{franks-misiurewicz-conley}).
\begin{theorem} If $f\colon X\to X$ is a homeomorphism of a compact metric space, then there is a complete Lyapunov function $g\colon X\to \R$ for $f$.
\end{theorem}

Now suppose $f\colon S\to S$ is a homeomorphism of the compact surface $S$.
Given a fixed $\epsilon>0$, we say that $x,y\in \CR(f)$ are $\epsilon$-related if there are $\epsilon$-pseudo orbits from $x$ to $y$ and from $y$ to $x$. This is an equivalence relation in $\CR(f)$; we call the equivalence classes $\epsilon$-transitive components. It is easy to see that there are finitely many $\epsilon$-transitive components \cite[Lemma 1.5]{franks-realizing}. Moreover, in \cite[Theorem 1.6]{franks-realizing} it is proved that every $\epsilon$-transitive component is of the form $g^{-1}([a,b])\cap \CR(f)$ for some complete Lyapunov function $g$ for $f$ and $a,b\in \R$ are regular values. Note that $\epsilon$-transitive components are compact and invariant.

\begin{proposition} \label{pro:isolating} Let $\Lambda$ be an $\epsilon$-transitive component for some $\epsilon>0$. Then there are compact surfaces with boundary $M_1\subset M_2\subset S$ such that $f(M_i)\subset \inter M_i$ for $i=1,2$, and 
$$\Lambda = (M_2\sm M_1)\cap \CR(f).$$
\end{proposition}
\begin{proof} Let $g$ be a complete Lyapunov function for $f$ such that the $\epsilon$-transitive component $\Lambda$ verifies that  $\Lambda=g^{-1}([a,b]) \cap \CR(f)$ for some regular values $a<b$.
Consider a function $\til{g}$ which  coincides with $g$ in a neighborhood of $\CR(f)$ and is $C^1$ in a neighborhood of $\{a,b\}$. If $\til{g}$ is $C^0$-close enough to $g$, it will be a complete Lyapunov function for $f$. Choose regular values $a'<b'$ (in the differentiable sense) for $\til{g}$ such that there are no points of $g(CR(f))$ between $a$ and $a'$ or between $b$ and $b'$, and define $M_b = \til{g}^{-1}((-\infty,b'])$, $M_a=\til{g}^{-1}((-\infty,a'])$. It is easy to verify that $M_a$ and $M_b$ satisfy the required properties (see Figure \ref{fig:reducing}).
\end{proof}
\begin{figure}
\begin{center}
\includegraphics[height=5cm]{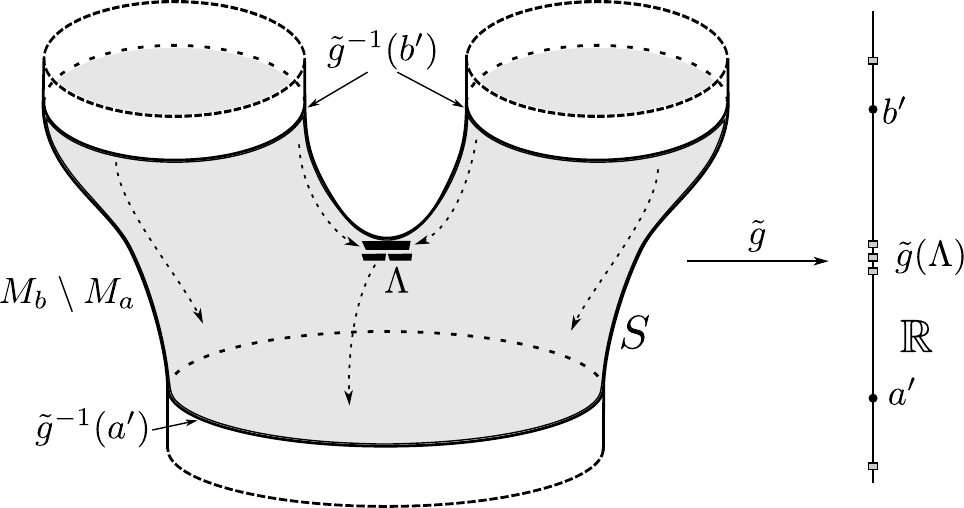}
\caption{Example situation in the proof of Proposition \ref{pro:isolating}}
\label{fig:reducing}
\end{center}
\end{figure}

\subsection{Reducing neighborhoods for $\epsilon$-transitive components}

First we recall some definitions. If $T$ is a non-compact surface, a \emph{boundary representative} of $T$ is a sequence $P_1\supset P_2\supset\cdots$ of connected unbounded (\ie not relatively compact) open sets in $T$ such that $\bd_T P_n$ is compact for each $n$ and for any compact set $K\subset T$, there is $n_0>0$ such that $P_n\cap K=\emptyset$ if $n>n_0$ (here we denote by $\bd_T P_n$ the boundary of $P_n$ in $T$). Two boundary representatives $\{P_i\}$ and $\{P_i'\}$ are said to be equivalent if for any $n>0$ there is $m>0$ such that $P_m\subset P_n'$, and vice-versa. The \emph{ideal boundary} $\mathrm{b_I}T$ of $T$ is defined as the set of all equivalence classes of boundary representatives. We denote by $\hat{T}$ the space $T\cup \mathrm{b}_I T$ with the topology generated by sets of the form $V \cup V'$, where $V$ is an open set in $T$ such that $\bd_T V$ is compact, and $V'$ denotes the set of elements of $\mathrm{b_I} T$ which have some boundary representative $\{P_i\}$ such that $P_i\subset V$ for all $i$. We call $\hat{T}$ the \emph{ends compactification} or \emph{ideal completion} of $T$.

Any homeomorphism $f\colon T\to T$ extends to a homeomorphism $\hat{f}\colon \hat{T}\to \hat{T}$ such that $\hat{f}|_{T} = f$. If $\hat{T}$ is orientable and $\mathrm{b_I} T$ is finite, then $\hat{T}$ is a compact orientable boundaryless surface. See \cite{richards} and \cite{ahlfors-sario} for more details.

The following lemma states that we can see an $\epsilon$-transitive component $\Lambda$ as a subset of an $f$-invariant open subset of $S$ such that the chain recurrent set of the extension of $f$ to its ends compactification consists of finitely many attracting or repelling periodic points together with the set $\Lambda$. 

\begin{lemma} \label{lem:reducing} Let $S$ be a compact orientable surface, and let $f\colon S\to S$ be a homeomorphism. If $\Lambda$ is an $\epsilon$-transitive component, then there is an open invariant set $T\subset S$ with finitely many ends such that each end is either attracting or repelling and $\CR(f|_T) = \Lambda$. 
\end{lemma}

\begin{proof} 
Let $M_1\subset M_2$ be the compact surfaces with boundary given by Proposition $\ref{pro:isolating}$. Removing some components of $M_1$ and $M_2$ if necessary, we may assume that every connected component of $M_i$ intersects $\CR(f)$ for $i=1,2$. Similarly, we may assume that every connected component of $S\sm \inter M_i$ intersects $\CR(f)$ for $i=1,2$ (removing some components of $S\sm M_i$ if necessary). This does not modify the properties of $M_i$ given by Proposition \ref{pro:isolating}.

 Note that $f(M_i)$ and $L_i=M_i\setminus \inter f(M_i)$  are compact surfaces with boundary whose union is $M_i$ and they intersect only at some boundary circles (at least one for each $M_i$, since neither $M_1$ nor $S\sm M_2$ are empty), so that their Euler characteristics satisfy $$\chi(M_i) = \chi(L_i) + \chi(f(M_i)).$$ 
But since $\chi(M_i) = \chi(f(M_i))$, it follows that $\chi(L_i)=0$. Thus, to show that $L_i$ is a union of annuli, it suffices to show that no connected component of $L_i$ is a disk (since that implies that the Euler characteristic of each connected component of $L_i$ is at most $0$).

Suppose that some (closed) disk $D$ is a connected component of $L_i$. Then the boundary $\bd D$ is a component of $\bd M_i \cup \bd f(M_i)$. 

Suppose first that $\bd D$ is a boundary circle of $M_i$. Since $D\subset L_i = M_i\sm \inter f(M_i)$, it follows that $D\subset M_i$. Thus, $D$ is a connected component of $M_i$. On the other hand, since $D\cap \inter f(M_i)=\emptyset$ and $f(M_i)\subset \inter M_i$ is disjoint from $\bd M_i\supset \bd D$, it follows that $D\cap f(M_i)=\emptyset$. Since $f^n(D)\subset f^n(M_i)\subset f(M_i)$ for $n\geq 1$, we can conclude that $D\cap \CR(f)=\emptyset$, because the distance from $f(M_i)$ to $D$ is positive, so that no $\epsilon$-pseudo orbit starting in $D$ can return to $D$ if $\epsilon$ is small enough. This contradicts the fact that every component of $M_i$ intersects $\CR(f)$ as we assumed in the beginning of the proof.

Now suppose that $\bd D$ is a component of $\bd f(M_i)$. Since $D\cap \inter f(M_i) = \emptyset$, it follows that $D$ is a connected component of $S\sm \inter f(M_i)$, so that $D'=f^{-1}(D)$ is a component of $S\sm \inter M_i$. But $f(D')\subset M_i$, which implies that $f^n(D')\subset M_i$ for all $n>0$. As before, this implies that $D'$ is disjoint from $\CR(f)$, contradicting the fact that every component of $S\sm \inter M_i$ intersects $\CR(f)$. This completes the proof that $L_i$ is a disjoint union of annuli.

Note that the number of (annular) components of $L_i$ coincides with the number of boundary components of $M_i$, since $L_i$ is a neighborhood of $\bd M_i$ in $M_i$.
The previous argument also shows that $M_i\sm \inter f^n(M_i)$ is a union of the same number of annuli if $n>0$. In fact, $M_i\sm \inter f^n(M_i) = \cup_{k=1}^{n-1} f^k(L_i)$, and the union is disjoint (modulo boundary). Thus the sets
$$\til{L}_1=\bigcup_{n\geq 0} f^n(L_1)\text{ and } \til{L}_2 = \bigcup_{n<0} f^n(L_2)$$
are increasing unions of annuli sharing one of their boundary components, hence they are both homeomorphic to a disjoint union of sets of the form $\mathbb{S}^1\times [0,1)$. Moreover, $\bigcap_{i>0} f^n(\til{L}_1) = \emptyset = \bigcap_{n<0} f^n(\til{L}_2)$. Let $k_i$ be the number of components of $\til{L}_i$ (or, which is the same, the number of boundary components of $M_i$).

Let $N=M_2\sm \inter M_1$ and write 
$$T=\til{L}_1\cup N \cup \til{L}_2.$$
It is easy to check that $f(T)=T$. Moreover, $\bd N = \bd \til{L}_1 \cup \bd{L}_2$, so $T$ is an open surface with $k_1+k_2$ ends. If we denote by $\hat{T}$ the ends compactification of $T$, and by $\hat{f}$ the extension of $f$ to $\hat{T}$ (which is a homeomorphism), we have that $\hat{f}$ has exactly $k_1+k_2$ periodic points, which are the ends of $T$. The ends in $\til{L}_1$ give rise to periodic attractors, and the ones in $\til{L}_2$ to periodic repellers. Since $\CR(f|_{\til{L}_i}) = \emptyset$ for $i=1,2$ and $\CR(f|_N)=\Lambda$, the surface $T$ has the required properties.
\end{proof}

\subsection{Aperiodic $\epsilon$-transitive components}

We now show that if an $\epsilon$-transitive component $\Lambda$ has no periodic points, then the neighborhood $T$ of $\Lambda$ in the statement of Lemma \ref{lem:reducing} is a disjoint union of annuli.

\begin{lemma} \label{lem:trace} Let $A$ be a matrix in $\SL(m,\Z)$. Then there is $n>0$ such that $\tr{A^n}\geq m$.
\end{lemma}
\begin{proof}
Let $r_1e^{2\pi i\theta_1},\dots, r_me^{2\pi i\theta_m}$ be the eigenvalues of $A$. Given $\epsilon>0$, we can find an arbitrarily large integer $n$ such that  $(2n\theta_1,\dots, 2n\theta_m)$ is arbitrarily close to a vector of integer coordinates, so that $\cos(4n\pi \theta_k)>1-\epsilon$. Thus $\tr{A^{2n}}=\sum_k r_k^n \cos(4n\pi \theta_k)>(1-\epsilon)\sum_k r_k^{2n}$. If $r_1=\cdots =r_m = 1$, then $\tr{A^{2n}} \geq m(1-\epsilon)$, and since $\tr{A^{2n}}$ is an integer, if $\epsilon$ was chosen small enough this implies that $\tr{A^{2n}} = m$. Now, if some $r_k\neq 1$, choosing a different $k$ we may assume $r_k>1$, so that $\tr{A^{2n}} >(1-\epsilon)r_k^{2n}$, and if $n$ is large enough and $\epsilon<1$ this implies that $\tr{A^{2n}}>m$. 
\end{proof}

\begin{theorem}\label{th:annuli}  Let $S$ be a compact orientable surface, and let $f\colon S\to S$ be a homeomorphism. If $\Lambda$ is an $\epsilon$-transitive component without periodic points, then either $S=\T^2$ and $f$ has no periodic points, or there is a disjoint union of periodic annuli $T\subset S$ such that the ends of each annulus are either attracting or repelling and $\CR(f|_T) = \Lambda$. 
\end{theorem}
\begin{proof}

Let $T$ be the surface given by Lemma \ref{lem:reducing}. If some component $T'$ of $T$ has no ends at all, then $T'$ is compact, and since it has no boundary, $T=T'=S$. Since $f$ has no periodic points in $\Lambda = \CR(f|_T) = \CR(f)$, it follows that there are no periodic points at all. The only compact boundaryless `surface admitting homeomorphisms without periodic points is $\T^2$, so $S=\T^2$ as required.

Now suppose $T$ has at least one end in each connected component, and let $\hat{T}$ be its ends compactification. 
Replacing $\hat{f}$ by some power of $\hat{f}$, we may assume that $\hat{f}$ preserves orientation, the periodic points arising from the ends of $\hat{T}$ are fixed points, and there are no other periodic points. Moreover, each connected component of $\hat{T}$ is invariant, all fixed points in $\hat{T}$ are attracting or repelling and there is at least one in each connected component. Thus we may (and will) assume from now on that $\hat{T}$ is connected and we will show that it is a sphere with exactly two fixed points, so that the corresponding connected component of $T$ is an annulus as desired.

Since the fixed points of $\hat{f}$ are attracting or repelling, the index of each fixed point is $1$. Since there are no other periodic points, the same is true for $\hat{f}^n$, for any $n\neq 0$. Thus we get, from the Lefschetz-Hopf theorem,
$$L(\hat{f}^n) = \#\Fix(\hat{f}^n) = \#\Fix(\hat{f})$$
where $L(f)$ denotes the Lefschetz number of $f$ (see \cite{franks-misiurewicz-conley}), defined by
$$L(\hat{f}^n)=\operatorname{tr}(\hat{f}_{*0})-\operatorname{tr}(\hat{f}_{*1})+\operatorname{tr}(\hat{f}_{*2}),$$ where  $\hat{f}_{*i}$ is the isomorphism induced by $\hat{f}$ in the $i$-th homology $H_i(\hat{T}, \mathbb{Q})$. 

It is clear that $\operatorname{tr}(\hat{f}_{*0})=1$ because we are assuming that $\hat{T}$ is connected. Since $\hat{T}$ is orientable and we are assuming that $\hat{f}$ preserves orientation, and from the fact that $\hat{T}$ is a closed surface, we also have that $\operatorname{tr}(\hat{f}_{*2})=1$.
Thus 
$$1\leq \#\Fix(\hat{f}) = L(\hat{f}^n) = 2 - \operatorname{tr}(A^n),$$
where $A$ is a matrix that represents $\hat{f}_{*1}$. Since $\hat{f}$ is a homeomorphism, $A\in \SL(\beta_1,\Z)$, where $\beta_1$ is the first Betti number of $\hat{T}$. By Lemma \ref{lem:trace} we can find $n$ such that $\operatorname{tr}(A^n)\geq \beta_1$. It follows that $\beta_1\leq  1$. But since $\hat{T}$ is a closed orientable surface, $\beta_1$ is even, so that $\beta_1=0$. That is, the first homology of $\hat{T}$ is trivial. We conclude that $\hat{T}$ is the sphere.  Since $\hat{f}$ preserves orientation, this means that $L(\hat{f})=2$, so that there are exactly two fixed points as we wanted to show. 
\end{proof}

\subsection{Proof of Theorem \ref{th:A}}

\begin{proof}[Proof of Theorem \ref{th:A}] Suppose there exists $\epsilon>0$ and some $\epsilon$-transitive component which has no periodic points. By Theorem \ref{th:annuli}, we have two possibilities: First, $S=\T^2$ and there are no periodic points in $S$. But this is not possible due to Corollary \ref{coro:torus}. The second and only possibility is that $\Lambda=\CR(f)\cap (A_1\cup\cdots\cup A_m)$ where the union is disjoint, each $A_i$ is a periodic open annulus, and each end of $A_i$ is either attracting or repelling. Using $f^n$ instead of $f$, we may assume that each $A_i$ is invariant. Let $\Lambda_i = \Lambda\cap A_i$. Then $A_i$ is an open invariant annulus such that $\CR(f|_{A_i}) = \Lambda_i$ and $f$ has no periodic points in $A_i$. But it is easy to see that $f$ has the shadowing property in $\CR(f|_{A_i})$, and this contradicts Corollary \ref{coro:annulus}.
\end{proof}

\subsection{Kupka-Smale diffeomorphisms}

The proof of theorem \ref{th:MS} starts with the next theorem that holds in any dimension for any Kupka-Samle diffeomorphism having the shadowing property. 

\begin{theorem}\label{th:KS} Let $f\colon M\to M$ be a Kupka-Smale diffeomorphism of a compact manifold having the shadowing property. Suppose there is a chain transitive component $\Lambda$ which contains a periodic orbit $p$. Then $\Lambda$ is the homoclinic class of $p$.
\end{theorem}
\begin{proof} Suppose that $f^k(p)=p$ and there is some point $x\in \Lambda$ which is not in the orbit of $p$. Let $\epsilon>0$, and choose $\delta>0$ such that every $\delta$-pseudo orbit is $\epsilon$-shadowed by an orbit. Since $x$ and $p$ are in the same chain transitive component, we can find a $\delta$-pseudo orbit $x_{-a},\dots, x_b$ such that $x_0=x$, $x_{-a}=p$ and $x_b = p$. Define $x_{-n} = f^{-n+a}(p)$ for $n>a$ and $x_n = f^{n-b}(p)$ for $n>b$. Then  $\{x_n\}$ is a $\delta$-pseudo orbit, which is $\epsilon$-shadowed by the orbit of some $y\in M$, which is not in the orbit of $p$ if $\epsilon$ is small enough. 

Note that $d(f^{-kn}(y), f^a(p))<\epsilon$ if $n>a$ and $d(f^{kn}(y), f^{-b}(p))<\epsilon$ if $n>b$. If $\epsilon$ is small enough, this implies that $y\in W^u(f^{a}(p))\cap W^s(f^{-b}(p))$. Note that the latter intersection is transverse since $f$ is Kupka-Smale. Since $y$ is $\epsilon$-close to $x$ and $\epsilon$ was arbitrary, it follows that $x$ is in the homoclinic class of $p$. 

It is clear that any point in the homoclinic class of $p$ is in the same chain transitive component of $p$. This completes the proof.
\end{proof}

\begin{corollary} If $f\colon M\to M$ is Kupka-Smale, then chain transitive components contain at most one or infinitely many periodic orbits.
\end{corollary}

\begin{corollary} If $f\colon M\to M$ is Kupka-Smale, then either $f$ has positive entropy or every chain transitive component consists of a single periodic orbit.
\end{corollary}

Now we can prove Theorem \ref{th:MS}.

\begin{proof}[Proof of Theorem \ref{th:MS}] Since $f$ has finitely many periodic orbits, by Theorem $\ref{th:KS}$ each chain transitive component of $f$ contains at most one periodic orbit, and if it does it contains nothing else. We need to show that there are no chain transitive components without periodic orbits, as this would imply that $\CR(f)=\Per(f)$, and the Kupka-Smale condition then implies then that $f$ is Morse-Smale.

Suppose by contradiction that there is some chain transitive component $\Lambda$ without periodic points. Since there are finitely many periodic orbits, for $\epsilon>0$ small enough it holds that the $\epsilon$-transitive components of periodic points are disjoint from $\Lambda$. Thus the $\epsilon$-transitive component $\Lambda_0$ containing $\Lambda$ contains no periodic points. By Theorem \ref{th:A}, this is a contradiction.
\end{proof}

\section{An example with an aperiodic class: proof of Theorem \ref{th:B}}
Let us briefly explain the idea of the construction of the example from Theorem \ref{th:B}. We will define a map $f$ on an arbitrary surface $S$, which has an invariant annulus homeomorphic to $C_1=\mathbb{S}^1\times [-1,1]$.
Our map will be such that the circle $C=\mathbb{S}^1\times\{0\}$ is invariant and $f|_C$ is an irrational rotation. This circle is going to be an aperiodic class. To guarantee that the system has the shadowing property we combine two ideas: first, we will make sure that $f$ is hyperbolic outside any neighborhood of $C$, and that $C$ has arbitrarily small ``traps'' (i.e. neighborhoods from which pseudo-orbits cannot escape). The traps guarantee that a pseudo-orbit that has a point far from $C$ is disjoint from a neighborhood of $C$, and a pseudo-orbit with a point close enough to $C$ must be entirely contained in a small neighborhood of $C$. For the first type of pseudo-orbits, the hyperbolicity of $f$ guarantees a shadowing orbit. In order to guarantee the shadowing of pseudo-orbits of the second type, \ie those remaining in a small neighborhood of $C$, we require that there is a sequence of hyperbolic sets of a special kind (``crooked horseshoes'') accumulating on the circle $C$. These sets have the property that they contain orbits whose first coordinates approximate increasingly well any $\epsilon$-pseudo-orbit of the rotation on $C$.

\subsection{Outline of the construction}\label{sec:exo}

In this section we will give a series of conditions that our map $f\colon S\to S$ must satisfy, and we will show how these conditions guarantee that $f$ has the shadowing property while exhibiting an aperiodic class. In later sections we proceed to construct $f$ satisfying the required conditions.

If $K\subset U\subset S$, let us say that $U$ is a \emph{forward (backward) trap} for $K$ if there exists $\delta>0$ such that any $\delta$-pseudo-orbit with initial point in $K$ remains entirely in $U$ in the future (resp. in the past); \ie if $\{z_n\}_{n\in \Z}$ is a $\delta$-pseudo-orbit and $z_0\in K$ then $z_n\in U$ for every positive (resp. negative) integer $n$. If $U$ is both a forward and backward trap, we say that $U$ is a \emph{full trap} for $K$.

Let us note that a if $U$ is a full trap for $K$, then there exists $\delta>0$ such that any $\delta$-pseudo-orbit intersecting $K$ lies entirely in $U$, and similarly any $\delta$-pseudo-orbit intersecting $S\sm U$ lies entirely in $S\sm K$ (in particular, $S\sm K$ is a full trap for $S\sm U$ as well).

If $W$ is an open relatively compact set containing $K$ and such that $f(\ol{W})\subset W$, then $W$ is a forward trap for $K$. Indeed, it suffices to choose $\delta< d(f(\ol{W}),S\sm W)$. Similarly, if $f^{-1}(\ol{W})\subset W$, then $W$ is a backward trap for $K$. 

Thus, in order to guarantee that a set $U$ is a full trap for $K\subset U$ it suffices to verify that there are relatively compact open sets $W_+, W_-$ such that $K\subset W_\pm\subset U$, $f(\ol{W}_+)\subset W_+$, and $f^{-1}(\ol{W}_-)\subset W_-$.

Suppose that there is a decreasing sequence of closed annuli $\{\aC_n\}_{n\in \N}$ in $S$, represented in appropriately chosen coordinates as $$\aC_n=\mathbb{S}^1\times [-c_n,c_n] \subset \A:=\mathbb{S}^1\times \R,$$ where $\{c_n\}_{n\in \N}$ is a decreasing sequence of positive numbers with $c_n\to 0$ as $n\to \infty$ (and $c_1=1$). 

Assume that:
\begin{enumerate}
\item[(1)] $f$ is an orientation-preserving diffeomorphism, and $f(\aC_n)=\aC_n$ for all $n\in \N$;
\item[(2)] the two circles bounding $\aC_n$ are normally hyperbolic, repelling if $n$ is odd and attracting if $n$ is even;
\item[(3)] if $C=\mathbb{S}^1\times \{0\}$, then $f|_C$ is an irrational rotation;
\item[(4)] for each $n\in \N$, $f|_{\aC_1\sm \inter \aC_n}$ is Axiom A with strong transversality.
\item[(5)] $f|_{S\sm \inter{\aC_1}}$ is Morse-Smale.
\end{enumerate}

Note that (3) implies that if $n$ is even, one may find a neighborhood $W$ of $\aC_n$ arbitrarily close to $\aC_n$ (i.e. in the $\epsilon$-neighborhood of $\aC_n$ for any fixed $\epsilon>0$) such that $f(\ol{W})\subset W$, and similarly if $n$ is odd one may find a similar neighborhood such that $f^{-1}(\ol{W})\subset W$.

Due to our previous observations, this implies that $\aC_{k}$ is a full trap for $\aC_{k+2}$ for any $k\in \N$. In other words, there is $\delta>0$ (depending on $k$) such that any $\delta$-pseudo-orbit intersecting $\aC_{k+2}$ is disjoint from $\aC_{k}$, and likewise, any $\delta$-pseudo-orbit intersecting $\aC_{k}$ never visits $S\sm \aC_{k+2}$. 

If $k\in \N$ is fixed, by (4) we know that $f|_{\aC_1\sm \inter \aC_{k+2}}$ is Axiom A with strong transversality, and by (5) so is $f|_{S\sm \inter \aC_1}$. These two facts, together with the fact that $\bd \aC_1$ is normally hyperbolic, imply that $f|_{S\sm \inter \aC_{k+2}}$ is Axiom A with strong transversality, and as mentioned in the introduction this guarantees that $f|_{S\sm \inter \aC_{k+2}}$ has the shadowing property (see \cite{pilyugin}). Hence, given $\epsilon>0$, there exists $\delta>0$ such that any $\delta$-pseudo-orbit of $f$ lying entirely in $S\sm \inter \aC_{k+2}$ is $\delta$-shadowed by an orbit of $f$. 

But since $\aC_{k}$ is a full trap for $\aC_{k+2}$, after reducing $\delta$ if necessary we know that any $\delta$-pseudo-orbit intersecting $S\sm \aC_k$ is entirely contained in $S\sm \aC_{k+2}$, and so it is $\epsilon$-shadowed by an orbit of $f$. 
In other words, we have seen that for any given $\epsilon>0$ and $k\in \N$, there is $\delta>0$ such that a $\delta$-pseudo-orbit which is \emph{not} $\epsilon$-shadowed by an orbit of $f$ must lie entirely in $\aC_k$.

Therefore, in order for $f$ to have the shadowing property, we need a condition that guarantees that if $k\in \N$ is large enough and $\delta>0$ is small enough, then any $\delta$-pseudo-orbit lying entirely in $\aC_k$ is $\epsilon$-shadowed by an orbit.
The assumption that we add for that purpose (using the coordinates of $\A$) is the following. Let $p\colon \aC_1\to C$ denote the map $(x,y)\mapsto (x,0)$. 
\begin{enumerate}
\item[(6)] For each $\epsilon'>0$ and $k\in \N$, there is $\delta'>0$ such that any $\delta'$-pseudo-orbit of $f|_C$ is $\epsilon'$-shadowed by the first coordinate of some orbit of $f$ in $\aC_k$; \ie if $\{x_n\}_{n\in \N}$ is a $\delta'$-pseudo-orbit of $f|_C\colon C\to C$, then there is $z\in \aC_k$ such that $d(p(f^n(z)), x_n)<\epsilon'$ for all $n\in \Z$.

\end{enumerate}
Let us explain why this condition suffices.
Fix $\epsilon>0$, choose $\delta'$ as in (6) using $\epsilon'=\epsilon/3$, and let $\delta=\delta'/3$. 
Choosing $k$ large enough, we may assume that $c_k<\epsilon/3$ and, moreover, $d(f(z), f(p(z)) < \delta$ whenever $z\in \aC_k$. Suppose that $\{z_n\}_{n\in \Z}$ is a $\delta$-pseudo-orbit for $f$ contained entirely in $\aC_k$. Then $\{p(z_n)\}_{n\in \Z}$ is a $\delta'$-pseudo-orbit for $f|_C$. By (6), there is $z\in \aC_k$ such that $d(p(f^n(z)), p(z_n))<\epsilon/3$ for all $n\in \Z$. This implies that $d(f^n(z), z_n)<\epsilon/3+2c_k <\epsilon$ for all $n\in \Z$, so that $\{z_n\}$ is $\epsilon$-shadowed by the orbit of $z$, as we wanted.

Therefore, a diffeomorphism satisfying (1)-(6) has the shadowing property, and one easily verifies that $C$ is an aperiodic chain transitive class.

To obtain a diffeomorphism satisfying (1)-(6), the main problem is finding a diffeomorphism from a neighborhood of $\aC_1$ in $\A$ to a neighborhood of $\aC_1$ in $\A$ satisfying (1)-(4) and (6). Indeed, once we have found such $f$, using standard arguments (as in \S\ref{sec:construct}) we may modify an arbitrary Morse-Smale diffeomorphism of $S$ in a neighborhood $D$ of a hyperbolic repelling fixed point such that $\ol{D}\subset f(D)$, replacing the dynamics in $D$ by an attracting fixed point and an invariant annulus with the dynamics of $f|_{\aC_1}$ (see Figure \ref{fig:embed-MS}). 

\begin{figure}
\begin{center}
\includegraphics[height=6cm]{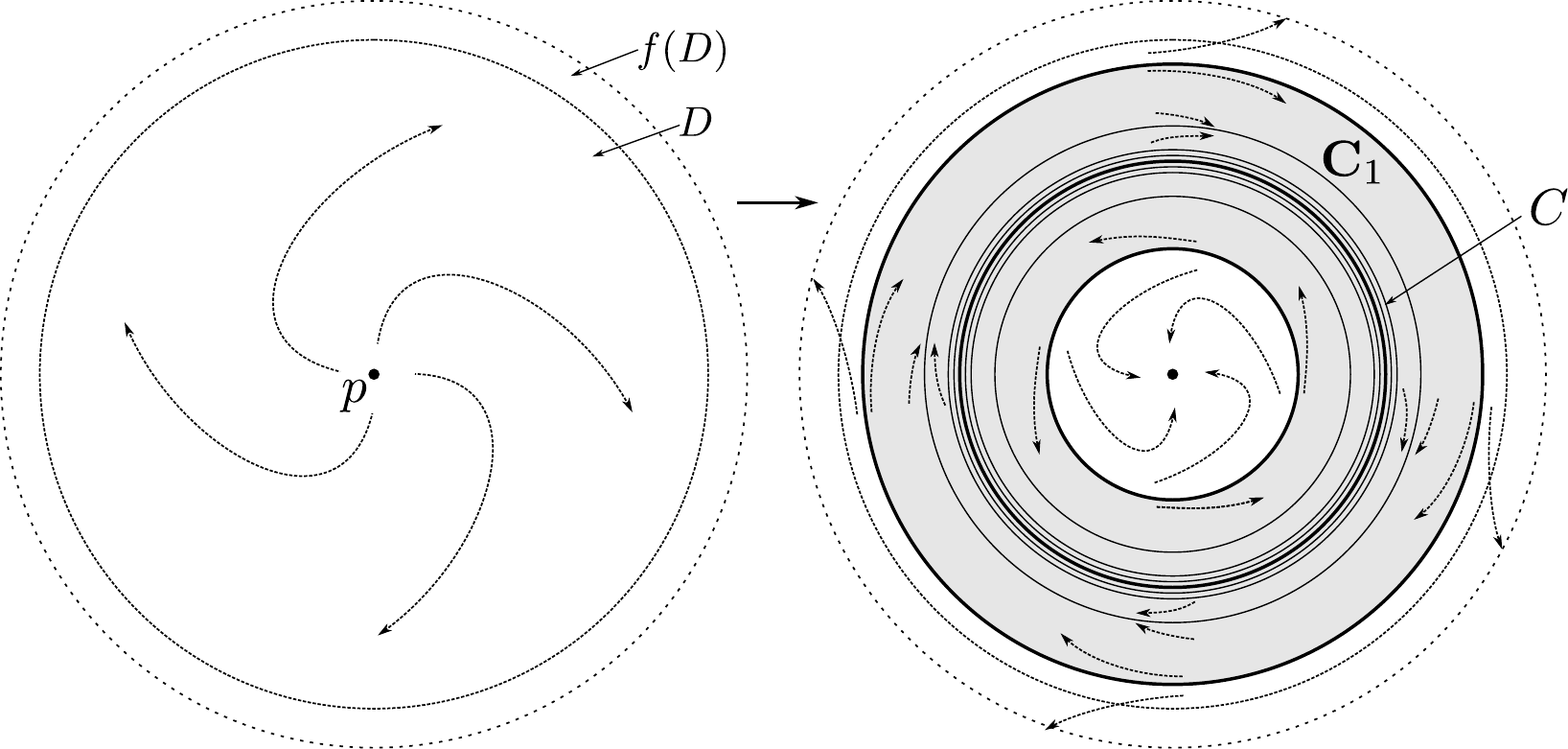}
\caption{The Morse-Smale diffeomorphism is modified only in $D$}
\label{fig:embed-MS}
\end{center}
\end{figure}

Thus, from now on, we are devoted to constructing the required map in $\aC_1$. The main difficulty is guaranteeing condition (6). For that purpose, we will use a sequence of hyperbolic sets accumulating on $C$ which have increasingly fine Markov partitions allowing to codify the required sequence of ``jumps'' for the first coordinates. In the next section we describe a model hyperbolic set which will be used for that purpose.

\subsection{Crooked horseshoes}

We begin by defining a diffeomorphism $H\colon \ol{\A} \to \ol{\A}$ of the closed annulus $\ol{\A}=\mathbb{S}^1\times [0,1]$ which has a ``crooked horseshoe'' wrapping around the annulus. Such a map is obtained by mapping a closed annulus to its interior as in figure \ref{fig:horseshoe}. The regions $A$ and $B$ are rectangles in the coordinates of $\ol{\A}$. The region $A$ is mapped to the interior of $A$, and $B$ is mapped to the gray region, which intersects $B$ in five rectangles. The map $H$ is contracting in $A$, while in $B\cap H^{-1}(B)$ it contracts in the radial direction and expands in the ``horizontal'' direction in a neighborhood of $B$ (affinely). 

This defines a diffeomorphism $H$ from $\ol{\A}$ to the interior of $\ol{\A}$. It is easy to see that the nonwandering set of $H$ consists of two parts: the set $K_0$, which is the maximal invariant subset of $H$ in $B$ and an attracting fixed point $p$ in $A$.

\begin{figure}
\begin{center}
\includegraphics[height=6cm]{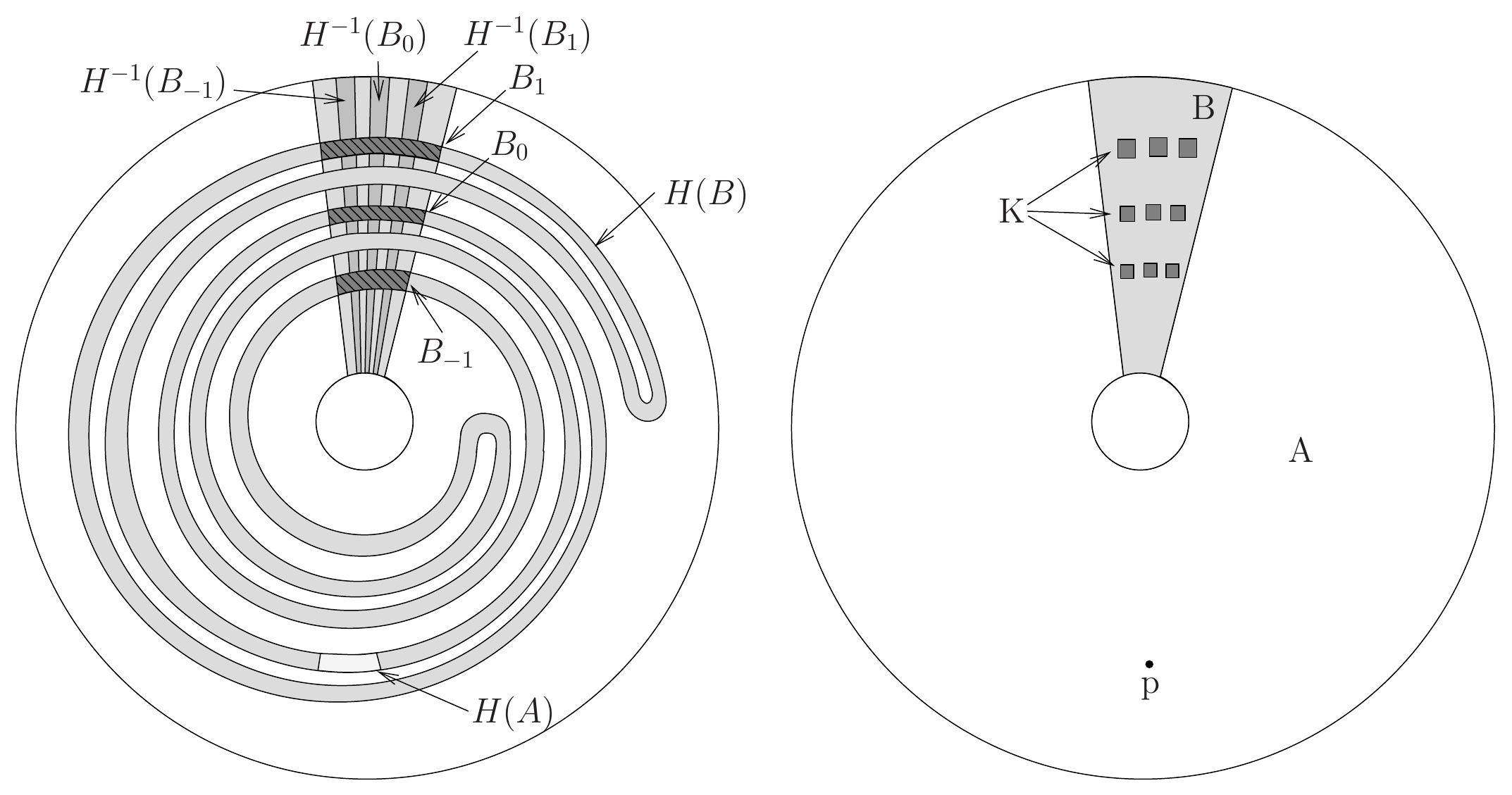}
\caption{A crooked horseshoe}
\label{fig:horseshoe}
\end{center}
\end{figure}

Since $H$ is affine in a neighborhood of $K_0$, the set $K_0$ is hyperbolic.
We can regard $H$ as a diffeomorphism from $\ol{\A}$ to itself by doing the above construction inside a smaller annulus, and then extending $H$ to the boundary in a way that the two boundary components are repelling and the restriction of $H$ to the boundary is Morse-Smale. In this way we obtain an Axiom A diffeomorphism $H\colon \ol{\A}\to \ol{\A}$.

As in the classical horseshoe, we have a natural Markov partition consisting of the five rectangles of intersection of $B$ with $H(B)$, which induces a conjugation of $H|_{K_0}$ to a full shift on five symbols. However we will restrict our attention to the set $K\subset K_0$ which is the maximal invariant subset of $H$ in $B_{-1}\cup B_0 \cup B_1$ (\ie three particular rectangles of the Markov partition, which are chosen as in Figure \ref{fig:horseshoe}). For these, we have a conjugation of $H|_{K}$ to the full shift on three symbols $\sigma\colon \{-1,0,1\}^\Z \to \{-1,0,1\}^\Z$, where the conjugation $\phi\colon K\to \{-1,0,1\}^\Z$ is such that $\phi(z) = (i_n)_{n\in \Z} \iff H^n(z)\in K_{i_n}$ for all $n\in \Z$, where $K_i=K\cap B_i$. Note that if $z\in K_i$ then $f(z)$ turns once around the annulus clockwise if $i=1$ and counter-clockwise if $i=-1$, and $f(z)$ does not turn if $i=0$. This is clearly seen considering the lift $\hat{B}$ of $B$ to the universal covering of $\ol{A}$ (\ie a connected component of the pre-image of $B$ by the covering projection), and a lift $\hat{H}\colon \R\times [0,1]\to \R\times [0,1]$ of $H$ such that $\hat{H}$ has a fixed point in $\hat{B}$ (see figure \ref{fig:lift-horse}). If $\hat{B}_i$ are the lifts of the sets $B_i$ inside $\hat{B}$ and $\hat{K}$ is the part of $\hat{B}$ that projects to $K$, we have that $\hat{H}(\hat{z})\in \hat{B} + (i,0)$ if $\hat{z}\in \hat{B}_i\cap \hat{K}$.

\begin{figure}
\begin{center}
\includegraphics[height=3.6cm]{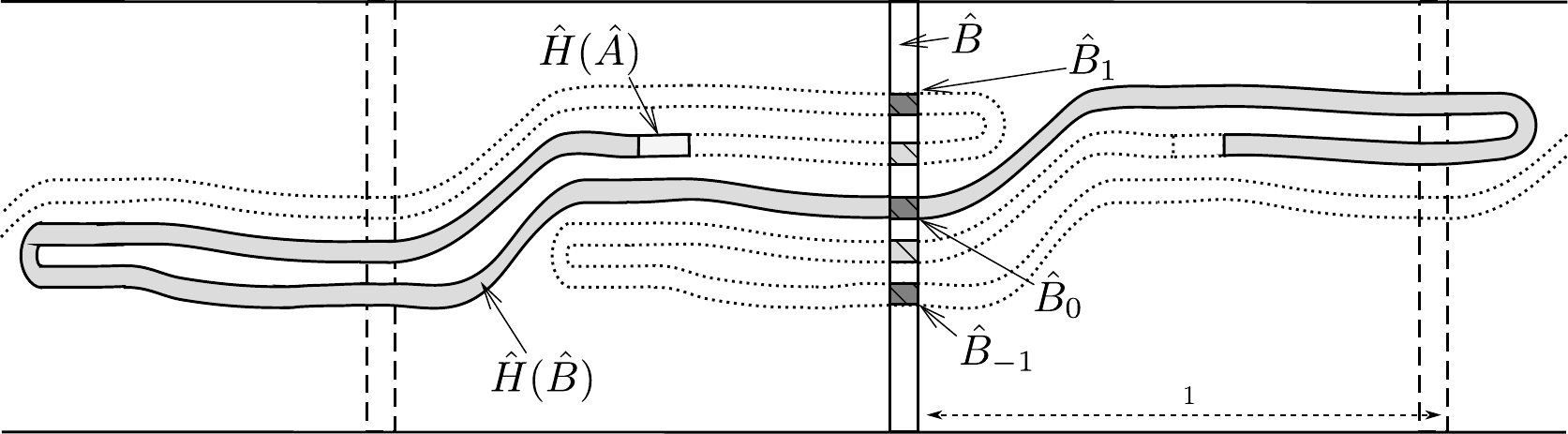}
\caption{The lift of $H$}
\label{fig:lift-horse}
\end{center}
\end{figure}

\begin{definition}\label{def:crooked}
We say that a diffeomorphism $f\colon \ol{\A}\to \ol{\A}$ has a crooked horseshoe if there is a hyperbolic invariant set $K\subset \ol{\A}$ with the properties described above; that is, there is a lift $\hat{f}$ of $f$ to the universal covering and three sets $\hat{K}_{-1}, \hat{K}_0$ and $\hat{K}_1$, which project to a Markov partition $K=K_{-1}\cup K_0\cup K_1$ such that $\hat{f}(\hat{z})\in i+\hat{K}$ if $\hat{z}\in \hat{K}_i$ (where $\hat{K} = \hat{K}_{-1}\cup \hat{K}_0\cup \hat{K}_1)$. The symbolic dynamics associated to the given partition of $K$ corresponds to a full shift on three symbols. 
Furthermore, the width of $\hat{K}$ (that is, the diameter of its projection to the first coordinate) is assumed to be at most $1$.
\end{definition}

\subsection{Approximations of a Liouville rotation with crooked horseshoes}

One of the key steps for our construction requires a hyperbolic diffeomorphism $f$ of the annulus which is $C^\infty$-close to the identity such that some power of $f$ has a crooked horseshoe. This is guaranteed by the next proposition, the proof of which is given in Appendix \ref{sec:crooked}. For $r\in \N$ and two diffeomorphisms $f$ and $g$, denote by $d_{C^r}(f,g)$ the $C^r$-distance between $f$ and $g$. 

\begin{proposition} \label{pro:flow} There exists a sequence of constants $\{C_r\}_{r\in \N}$ such that for any $m\in \N$ there is a $C^\infty$-diffeomorphism $H\colon \ol{\A}\to \ol{\A}$ isotopic to the identity, satisfying the following properties:
\begin{enumerate} 
\item $H$ is Axiom A with the strong transversality condition;
\item $H^m$ has a crooked horseshoe;
\item $d_{C^r}(H,\id) < C_r/m$ for each $r\in \N$.
\end{enumerate}
Furthermore, the lift of $f = H^m$ in Definition \ref{def:crooked} may be chosen as $\hat{H}^m$, where $\hat{H}$ is a lift of $H$.
\end{proposition}

Write $\ol{R}_\alpha\colon \ol{\A}\to \ol{\A}$ for the rotation $x\mapsto x+\alpha$, and $R_\alpha\colon \mathbb{S}^1\to \mathbb{S}^1$ for the analogous rotation of the circle. For the next proposition we use the following terminology: if $\{x_n\}_{n\in \N}$ is a sequence of points of $\mathbb{S}^1$ and $h\colon \ol{\A}\to \ol{\A}$ is a homeomorphism, we say that $\{x_n\}_{n\in \N}$ is $\epsilon$-shadowed by the first coordinate of the orbit of $z\in \ol{\A}$ if $d(h^n(z)_1, x_n)<\epsilon$ for all $n\in \Z$, where $(\cdot)_1$ denotes the first coordinate in $\ol{\A}=\mathbb{S}\times [0,1]$.

\begin{proposition} \label{pro:chato} For any $\epsilon>0$ and any Liouville number $\alpha$, there is an Axiom A diffeomorphism $h\colon \ol{\A}\to \ol{\A}$ satisfying the strong transversality condition and there is $\delta>0$ such that every $\delta$-pseudo orbit of $R_\alpha$ is $\epsilon$-shadowed by the first coordinate of some orbit of $h$. The map $h$ may be chosen arbitrarily $C^\infty$-close to $\ol{R}_\alpha$.
\end{proposition}

\begin{remark} The Liouville condition is necessary in order to control the $C^r$ distances for large values of $r$, because the proof involves regarding the $\delta$-pseudo-orbit $\{x_n\}$ of $R_\alpha$ as a pseudo-orbit of a rational rotation $R_{p/q}$. The `jump' size of $\{x_n\}$ regarded as a pseudo-orbit of $R_{p/q}$ is $\delta + |\alpha-p/q|$, and we need this size to be small when compared to $1/q^r$ due to a re-scaling that appears in the construction (see (\ref{eq:coverformula}) and (\ref{eq:bound1}) in the proof that follows). We may choose $\delta$ as small as desired, but the necessary bound on $|\alpha-p/q|$ depends precisely on $\alpha$ being Liouville. Note however that if one only wants $h$ to be $C^0$-close to $\ol{R}_{\alpha}$, simpler estimates allow to prove the same result for an arbitrary irrational $\alpha$ (due to the fact that the $q^r$ factor disappears).
\end{remark}

\begin{proof}[Proof of Proposition \ref{pro:chato}] Fix $\epsilon>0$ and a Liouville number $\alpha$. We need to construct, for any $r\in \N$, a map $h$ with the desired properties which is $C^r$-close to the rotation by $\alpha$. Thus we fix $r\in \N$ and $\eta>0$ from now on, and we will find $h$ satisfying the required properties such that $d_{C^r}(h, \ol{R}_\alpha)<\eta$. We may without loss of generality assume that $\epsilon$ and $\eta$ are both smaller than $1$.

Note that, for $p\in \Z$ and $q\in \N$, any $\delta$-pseudo orbit of $R_\alpha$ is a $(\delta+|\alpha-p/q|)$-pseudo orbit of $R_{p/q}$. The fact that $\alpha$ is Liouville implies that one may choose $q\in \N$ arbitrarily large and $p\in \Z$ such that 
\begin{equation}\label{eq:liouville} 
\Big|\frac{p}{q}-\alpha\Big|<\frac{1}{q^{r+2}}.
\end{equation}
Since $q$ may be chosen arbitrarily large, we will assume that 
\begin{equation}\label{eq:liouville2} 
\frac{1}{q-1}<\frac{\epsilon\eta}{C_r+4}.
\end{equation}
where $C_r$ is the constant from Proposition \ref{pro:flow}.

Set $m=q^{r}(q-1)$, and let $H\colon \ol{\A}\to \ol{\A}$ be a diffeomorphism as in Proposition \ref{pro:flow}. Since $H^m$ exhibits a crooked horseshoe, there is a hyperbolic set $K$ for $H^m$, a lift $\hat{H}$ of $H$ to the universal covering $\R\times [0,1]\to \ol{\A}$, and a set $\hat{K} = \hat{K}_{-1}\cup \hat{K}_{0}\cup \hat{K}_1\subset \R\times [0,1]$ which projects onto $K$ such that:
\begin{itemize}
\item The sets $\hat{K}_{-1}, \hat{K}_0, \hat{K}_1$ project to $\ol{\A}$ onto sets $K_{-1}, K_0, K_1$ partitioning $K$.
\item $\hat{H}^m(z)\in \hat{K}+(i,0)$ if $z\in \hat{K}_i$;
\item for any sequence $(i_n)_{n\in \Z}$ such that $i_n\in \{-1,0,1\}$, there is $z\in K$ such that $H^{mn}(z)\in K_{i_n}$ for all $n\in \Z$.
\item $\hat{K}$ has (horizontal) width at most $1$.
\end{itemize}
Since $H^m$ is isotopic to the identity, $\hat{H}^m$ commutes with the covering transformations $z\mapsto z+(k,0)$, $k\in \Z$.

Consider the homeomorphism $\hat{f}\colon \R\times [0,1]\to \R\times [0,1]$ defined by
\begin{equation}\label{eq:coverformula}
\hat{f}(x,y) = \Big(\frac{1}{q}\hat{H}_1(qx,y),\, \hat{H}_2(qx,y)\Big),
\end{equation}
where $\hat{H}_1$ and $\hat{H}_2$ are the coordinate functions of $\hat{H}$.
Then $\hat{f}$ commutes with the translations $z\mapsto z+(k/q,0)$, $k\in \Z$. In particular, $\hat{f}$ is the lift of some diffeomorphism $f\colon \ol{\A}\to \ol{\A}$, which in addition commutes with the periodic rotation $(x,y) \mapsto (x,y)+(1/q,0)$ of $\ol{\A}$.

Note that $f$ is a lift of $H$ to the $q$-sheeted covering
\begin{align*}
\tau \colon \R/\Z\times [0,1] &\to \R/\Z\times [0,1]\\
(x,y)&\mapsto (qx,y).
\end{align*}
Finally, consider the map $h=f+(p/q,0)$. Since $h$ is also a lift of $H$ to the aforementioned finite covering (because $h$ differs from $f$ by a covering transformation of $\tau$), and since $H$ is an Axiom A diffeomorphism with strong transversality, it follows easily that $h$ itself is Axiom A with strong transversality. 

A direct computation using (\ref{eq:liouville2}), (\ref{eq:coverformula}), and our choice of $m$ gives the following bound on the $C^r$-distance from $h$ to the rotation $\ol{R}_{p/q}$:
\begin{equation}\label{eq:bound1}
d_{C^r}(h,\ol{R}_{p/q}) = d_{C^r}(f,\id) \leq q^r d_{C^r}(H,\id) \leq q^r\frac{C_r}{m} = \frac{C_r}{q-1}<\eta,
\end{equation}
which is one of the conditions we wanted $h$ to meet. 

It remains to show that for some $\delta>0$, every $\delta$-pseudo-orbit of $R_\alpha$ is $\epsilon$-shadowed by the first coordinate of some orbit of $h$. We begin by observing that
\begin{equation}\label{eq:2}
\sup \{d(h(z)_1, R_{p/q}(z_1)): z\in \ol{\A}\} \leq d_{C^0}(h, \ol{R}_{p/q}) \leq d_{C^r}(h,\ol{R}_{p/q}) < \frac{C_r}{mq}
\end{equation}
where $(\cdot)_1$ denotes the first coordinate. For future reference note that (\ref{eq:2}) implies, for all $k\in \N$
\begin{equation}\label{eq:2prime}
\sup \{d(h^k(z)_1,R_{p/q}^k(z_1)) : z\in \ol{\A}\} < k\frac{C_r}{mq}.
\end{equation}
Indeed, this follows from a straightforward induction using (\ref{eq:2}) and noting that 
$$h^{k+1}(z)_1-R_{p/q}^{k+1}(z_1)= h(h^k(z))_1 - R_{p/q}(h^k(z)_1) + R_{p/q}(h^k(z)_1-R_{p/q}^k(z_1)).$$

Our choice $\delta$ will be the following
\begin{equation}\label{eq:delta}
\delta=\frac{1}{mq}-\frac{1}{q^{r+2}}>0.
\end{equation} 
Let $\{x_n\}$ be a $\delta$-pseudo orbit of $R_\alpha$. By (\ref{eq:liouville}) and (\ref{eq:delta}), the same sequence is also a $1/(mq)$-pseudo orbit for $R_{p/q}$, and thus $\{x_{mn}\}$ is a $1/q$-pseudo orbit for $R_{mp/q} = \id$. This means that $d(x_{mn}, x_{mn+1})\leq 1/q$ for all $n\in \Z$. 

Chose a reference point $(a,b)\in \hat{K}$, and let $(a',b')$ be the projection of $(a/q, b)$ to $\ol{\A}$. For each $n\in \Z$, let $j_n\in \Z$ be such that $x_{mn}$ lies in the smallest interval joining $a'+j_n/q$ to $a'+(j_n+1)/q$ in $\mathbb{S}^1$. Since $d(x_{mn}, x_{m(n+1)})\leq 1/q$, it follows easily that $j_{n+1}-j_n \in \{-1, 0, 1\}\, (\mathrm{mod}\, q)$. In particular, we may choose $i_{n+1}\in \{-1,0,1\}$ such that $i_{n+1} = j_{n+1}-j_n\, (\mathrm{mod}\, q)$. Thus, letting $j_n' = i_0+i_1+\cdots+i_{n-1}$, we have $j_n' = j_n\, (\mathrm{mod}\, q)$.

From the symbolic dynamics of $K$ for $H^m$ mentioned at the beginning of the proof, we know that there is $(s,t)\in K$ such that $H^{mn}(s,t) \in K_{i_n}$ for all $n\in \Z$. This means that, if $(\hat{s},t)$ is the point of $\hat{K}$ which projects to $(s,t)$, then 
$$\hat{H}^{mn}(\hat{s},t)\in \hat{K}+(i_0+i_1+\cdots + i_{n-1}, 0) = \hat{K}+ (j_n'-j_0',0).$$
Thus, $$\hat{H}^{mn}((\hat{s},t)+(j_0',0))= \hat{H}^{mn}(\hat{s},t)+(j_0',0) \in \hat{K} + (j_n',0)$$ for each $n\in \N$. Observe that $(a,b)+(j_n',0)\in \hat{K}+(j_n',0)$ as well. Since $\hat{K}$ has horizontal width at most $1$, the first coordinates of $\hat{H}^{mn}((\hat{s},t)+(j_0',0))$ and $(a,b)+(j_n',0)$ differ by at most $1$. By (\ref{eq:coverformula}), this implies that the first coordinates of $\hat{f}^{mn}(\hat{s}/q+j_0'/q, t)$ and $(a/q, b) + (j_n'/q,0)$ differ by at most $1/q$, which in turn means that, if $z$ denotes the projection of $(\hat{s}/q+j_0'/q,t)$ to $\ol{\A}$, 
$$d(f^{mn}(z)_1,\, a'+ j_n'/q) \leq 1/q,$$
where $(\cdot)_1$ denotes the first first coordinate. But $a'+j_n/q = a'+j_n'/q\in \mathbb{S}^1$ (because $j_n=j_n'\, (\mathrm{mod}\, q)$, and from the definition of $j_n$ we know that $d(a'+j_n/q,\, x_{mn})<1/q$. Thus, $d(f^{mn}(z)_1,\, x_{mn}) <2/q$. Noting that $h^m = f^m+mp/q = f^m$ (because $mp/q\in \Z$), we conclude that
\begin{equation}\label{eq:forpower}
d(h^{mn}(z)_1,\, x_{mn})< 2/q.
\end{equation}

\begin{figure}[ht]
\begin{center}
\includegraphics[height=9cm]{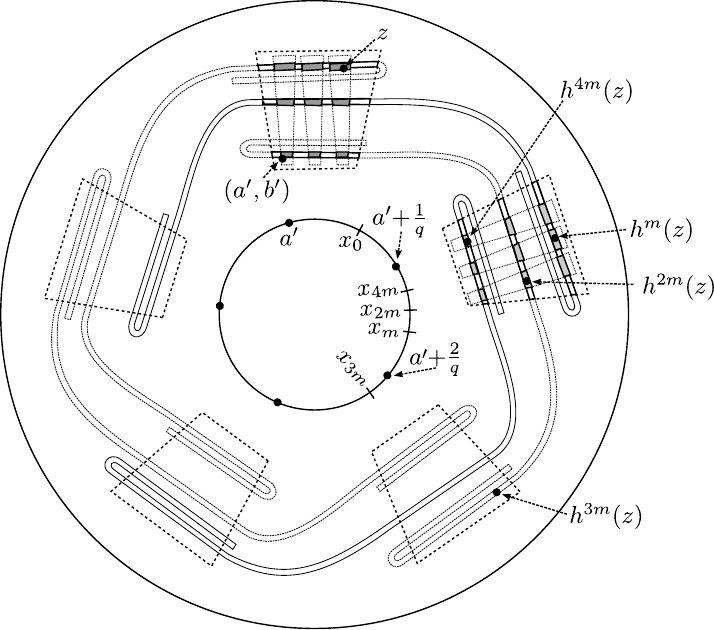}
\caption{$\{x_{mn}\}_{n\in \N}$ is a $1/q$-pseudo-orbit of $\id_{\mathbb{S}^1}$ shadowed by $\{h^{nm}(z)_1\}_{n\in \N}$.}
\label{fig:finitecovercrook}
\end{center}
\end{figure}
Figure \ref{fig:finitecovercrook} attempts to illustrate the situation with $h^m$ in $\ol{\A}$ (in the figure, $q=5$, and the 5 large blocks that are illustrated contain the preimages of $\hat{K}$ via the finite covering $\tau$; compare with Figures \ref{fig:horseshoe} and \ref{fig:lift-horse}Á).

From (\ref{eq:forpower}), (\ref{eq:2prime}), and the fact that $R_{p/q}^k$ is an isometry, follows that if $0\leq k\leq m-1$,  
\begin{align*}
d(h^{nm+k}(z)_1, x_{nm+k})&\leq d(h^k(h^{nm}(z))_1, R_{p/q}^k(h^{nm}(z)_1)) + d(h^{nm}(z)_1,x_{nm}) \\ &\hspace{1cm}+ d(R^k_{p/q}(x_{nm}),x_{nm+k})\\ &\leq \frac{kC_r}{mq} + \frac{2}{q}+\frac{k}{mq} \leq \frac{C_r+3}{q} < \eta\epsilon \frac{C_r+3}{C_r+4} < \epsilon
\end{align*}
where we also used (\ref{eq:liouville2}) and the fact that $\{x_n\}$ is a $1/(mq)$-pseudo orbit for $R_{p/q}$.
Thus $\{h^n(z)_1\}$ $\epsilon$-shadows $\{x_n\}$. This completes the proof.
\end{proof}

\subsection{Construction of $f$} \label{sec:construct}
We now construct a diffeomorphism of the cylinder $\aC_1= \mathbb{S}^1\times [-1,1]$ which satisfies properties (1)-(4) and (6) of \S\ref{sec:exo}, which concludes the proof of Theorem \ref{th:B}. 

To define $f$, fix a Liouville number $\alpha$ and denote by $R_\alpha\colon \A \to \A$ the map $R_\alpha(x,y)=(x+\alpha,y)$. We first choose a sequence of pairwise disjoint closed annuli $\{A_i:i> 0\}$ of the form $A_i=\mathbb{S}^1\times [a_i,b_i]$ with the following properties:
\begin{enumerate}
\item $b_1=1$, and $b_n>a_n>b_{n+1}>0$ for all $n\in \N$.
\item $b_n-a_n\to 0$ as $n\to \infty$;
\item\label{cond:3} $1/n^2\leq a_n-b_{n+1}\leq 1$ (this estimates the distance between $A_n$ and $A_{n+1}$).
\item $A_n$ converges to the circle $C=\mathbb{S}^1\times \{0\}$ as $n\to \infty$; (this follows from the previous items).
\end{enumerate} 
We will choose a sequence $h_n\colon A_n\to A_n$ of diffeomorphisms such that $h_n$ has the properties of the map $h$ of Proposition \ref{pro:chato} with the annulus $A_n$ instead of $\ol{\A}$, using $\epsilon = 1/n$. In other words, 
\begin{itemize}
\item $h_n$ is Axiom A with strong transversality;
\item There exists $\delta>0$ such that every $\delta$-pseudo-orbit of $R_\alpha$ is $1/n$-shadowed by the first coordinate of some orbit of $h$.
\end{itemize}
Moreover, we may choose $h_n$ such that 
\begin{equation}
\label{eq:ex1}
d_{C^n}(h_n,R_\alpha)<(2n)^{-2n}/K_n
\end{equation}
 where $K_n$ is a constant that we will specify later. Using $H^{-1}$ instead of $H$ in Proposition \ref{pro:flow}, we may obtain a map with the same properties as $h_n$ in Proposition \ref{pro:chato} such that its inverse is attracting instead of repelling on the boundary of the annulus. Thus we may assume that for $h_n$, the boundary of $A_n$ is repelling if $n$ is odd and attracting if $n$ is even. 

Moreover, from the proof of Proposition \ref{pro:flow}, it is possible to assume that the restriction of $H$ (and thus of $h_n$) to a neighborhood of the boundary components of $A_n$ has a simple dynamics, namely the product of a Morse-Smale diffeomorphism $g_n$ of the circle and a linear contraction or expansion; that is, 
$$h_n(x,y) = (g_n(x), L_n^\pm(y))$$
for $(x,y)$ in a neighborhood of $\bd^\pm A_n$, where $L_n^\pm(x,y) = \lambda_n(y-y_n^\pm)+y_n^\pm$.

Since the boundary of $A_n$ is attracting if $n$ is odd and repelling if $n$ is even, $\lambda_n>1$ and $\lambda_{n+1}<1$ or vice versa (we will assume the first case).

We define $f|_{A_n} = h_n$. To define $f$ in the regions between the $A_i$'s, let let $B_n$ be the annulus between $A_n$  and $A_{n+1}$. 
For $(x,y)\in B_n$, we define $f(x,y) = (f(x,y)_1, f(x,y)_2)$ using convex combinations:
\begin{align*} 
f(x,y)_1&=g_n(x)+\phi\left(\frac{y-y_n^-}{y_{n+1}^+-y_n^-}\right)(g_{n+1}(x)-g_n(x))\\
f(x,y)_2&= L_n^-(y) + \phi\left(\frac{y-y_n^-}{y_{n+1}^+-y_n^-}\right)(L_{n+1}^+(y)-L_n^-(y))
\end{align*}
where $\phi\colon \R\to \R$ is a fixed $C^\infty$ function such that $\phi(x) = 0$ if $x<0$, $\phi(x)=1$ if $x>1$ and $0\leq \phi(x)\leq 1$. We further asusme that $\phi$ is strictly increasing. Since $\phi$ is fixed, the constant $K_n$ that we used in the choice of $h_n$ can be chosen so that $\norm{\phi}_{C^n}<K_n$ and $K_n\geq 2$.

For convenience, let $t(y) =  \frac{y-y_n^-}{y_{n+1}^+-y_n^-}$ and $\Delta_n = y_{n+1}^+-y_n^-$. Note that $\abs{\Delta_n}$ is the distance from $A_n$ to $A_{n+1}$, so from condition (\ref{cond:3}) at the beginning of the proof we have
\begin{equation*}\label{eq:ex-delta} 1\geq \abs{\Delta_n}\geq 1/n^2\end{equation*}
Also note that from (\ref{eq:ex1}), 
$$d_{C^n}(g_n, g_{n+1})\leq 2(2n)^{-2n}/K_n,$$
so that if $0\leq i+j\leq n$ and $i,j\geq 0$, 
\begin{align*}
\abs{\frac{\partial^{i+j}}{\partial x^i \partial y^j}\left(\phi(t(y))(g_{n+1}(x)-g_n(x))\right)} &= 
\abs{\frac{\phi^{(j)}(t(y))}{\Delta_n^j}\left(g_{n+1}^{(i)}(x)-g_n^{(i)}(x)\right)}\\
&\leq 2/((2n)^{2n} \abs{\Delta_n^n} \leq 2/4^n
\end{align*}
and if $1\leq i\leq n$, using Leibniz's formula and the fact that $L_{n+1}^+$ and $L_n^-$ are affine maps, we find (again using (\ref{eq:ex1}))
\begin{multline*}
\abs{\frac{\partial^i}{\partial y^i}\left(\phi(t(y))(L_{n+1}^+(y)-L_n^-(y))\right)}\\= \abs{\frac{\phi^{(i)}(t(y))}{\Delta_n^i}\left(L_{n+1}^+(y) - L_n^-(y)\right) + n\frac{\phi^{(i-1)}(t(y))}{\Delta_n^{i-1}}(\lambda_{n+1}-\lambda_n)}
\\ \leq \frac{K_n}{\Delta_n^n}(2(2n)^{-2n}/K_n) + n\frac{K_n}{\Delta_n^n}(2(2n)^{-2n}/K_n)\\ \leq 2(n+1)/((2n)^{2n}\abs{\Delta_n^n})\leq 2(n+1)/4^n.$$
\end{multline*}
Putting these facts together, we see that
\begin{equation}\label{eq:crdist1}
d_{C^n}(f|_{B_n}, R_\alpha|_{B_n})\leq 2(n+1)/4^n\xrightarrow{n\to\infty} 0.
\end{equation}

\begin{figure}[ht]
\begin{center}
\includegraphics[height=5cm]{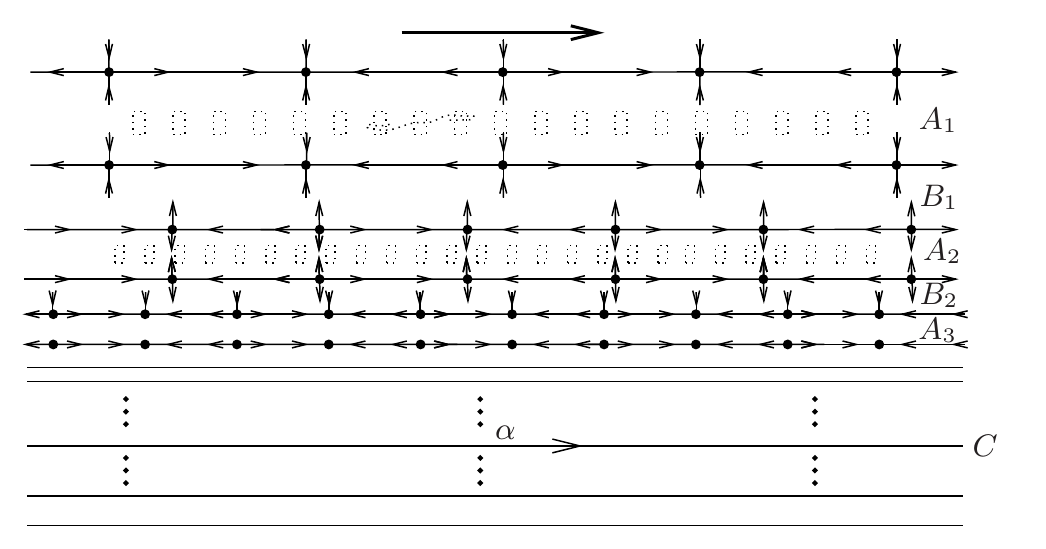}
\caption{The map $f$}
\label{fig:example}
\end{center}
\end{figure}

Note that the dynamics of $f$ in $B_n$ is trivial: the boundary components of $B$ are one attracting and the other repelling and there is no recurrence in the interior of $B$. In fact it is easy to see that $f(x,y)_2 - y$ is always positive or always negative for $(x,y)\in B_n$ (depending on the parity of $n$). 

To see that $f|_{B_n}$ is Axiom A, note that the nonwandering set of $f|_{B_n}$ is contained in the boundary of $B_n$ (which is in $A_{n+1}$ or $A_n$) so it consists of periodic points which are hyperbolic and finitely many. Moreover, since the boundary components of $B_n$ are one attracting and one repelling, a homoclinic intersection between saddles can only happen if the saddles are in different boundary components. A small perturbation supported in the interior of $B_n$ which does not affect our estimates ensures that all such intersections are transverse, guaranteeing the strong transversality condition. Note that $f$ is also Axiom A with strong transversality in each $A_n$. 

We have defined $f$ in $\mathbb{S}^1\times [0,1]$ (see Figure \ref{fig:example}). Repeating this procedure symmetrically in $\mathbb{S}^1\times [-1,0]$ and letting $f|_C = R_\alpha|_{C}$, we obtain a map $f\colon \aC_1=\mathbb{S}^1\times [-1,1] \to \aC_1$. By construction, $f$ is $C^\infty$ in $\aC_1\sm C$. To see that $f$ is also $C^\infty$ in $C$, it suffices to note that from (\ref{eq:ex1}) and (\ref{eq:crdist1}), 
$$d_{C^r}(f|_{\mathbb{S}^1\times (0,t]}, R_\alpha|_{\mathbb{S}^1\times (0,t]})\xrightarrow[t\to 0^+]{} 0,$$
and a similar fact replacing $(0,t]$ by $[-t,0)$, which easily implies that $f$ is $C^\infty$ in $C$ as well.

We note that $f$ is Kupka-Smale because it is Axiom A with strong transversaility in each of the invariant annuli composing $\aC_1$ (and there are no periodic points in $C$), and $f$ clearly satisfies the required properties (1)-(4) from \S\ref{sec:exo} using $\aC_k=\mathbb{S}^1\times [-b_k,b_k]$. From our choice of $f$ in each annulus $A_n$ (as a map $h_n$ obtained from Proposition \ref{pro:chato}) one sees that property (6) from \S\ref{sec:exo} also holds, completing the construction.

\section{Shadowing for one dimensional endomorphisms} 
\label{sec:onedim}

We now consider  smooth one dimensional endomorphisms on the circle (at least $C^2$) assuming that  the Holder shadowing property holds with Holder constant larger than $\frac{1}{2}.$
 We recall first some definitions. Let $f$ be a $C^r$ endomorphism of the circle.

\begin{definition}\label{holder} Let $f$ be a $C^r$ endomorphism of the circle and $\al\leq 1$,  $f$ has the $\al$-Holder shadowing property if there exists $C>0$ such that any $\e$-pseudo-orbit with $\e>0$ is $C\e^\al$-shadowed by an orbit. 
\end{definition}

\begin{definition}\label{expanding} A $C^r$ endomorphism of the circle  $f$ is expanding, if there exist $C>0$ and  $\la>1$ such that $\abs{(f^n)'(x)}>C\la^n$ for any $x$ in the circle and for any positive integer $n$. 
\end{definition}

\begin{definition}
\label{def:turning} A critical point $c$ (i.e., $f'(c)=0$) is a turning point if $c$ is either a local minimum or a local maximum of $f$.
\end{definition}

To prove Theorem \ref{crit-1} we will need the following result. 

\begin{theorem}\label{th:localhomeo} If $f$ is a transitive non-invertible local homeomorphism of the circle, then $f$ is topologically conjugate to a linear expanding map.
\end{theorem}

\begin{proof} Since $f$ is a local homeomorphism, if the degree of $f$ is $1$ or $-1$ it follows that $f$ is a homeomorphism, contradicting the assumption that $f$ is non-invertible. Thus the degreee $d$ of $f$ satisfies $\abs{d}\neq 1$. This implies (for example, see \cite[Prop. 2.4.9]{katok-book} that $f$ is semiconjugated to the linear expanding map $E_d\colon x\mapsto dx\, (\mathrm{mod}\, 1)$ via a monotone map $h$ of degree $1$; \ie $hf = E_d h$ and $h$ is a continuous surjection such that $h^{-1}(x)$ is a point or an interval for each $x$. Suppose $I=h^{-1}(x)$ is a nontrivial interval for some $x$. Then from the transitivity it follows that there is $k>0$ such that $f^k(I)\cap I\neq \emptyset$. Since $h(f^k(I)) = E_d^k h(I)$ and $h(I)=x$, we conclude that $E_d^k(x)=x$, and so $f^k(I)=I$. The transitivity of $f$ implies then that $\bigcup_{i=0}^{k-1} f^i(I)$ is dense in the circle, which in turn implies that $h$ has a finite image, a contradiction. This shows that $h^{-1}(x)$ is a single point for every $x$, so that $h$ is a homeomorphism and $f$ is conjugate to $E_d$.
\end{proof}

Before proceeding to the proof of Theorems \ref{crit-1} and \ref{crit-robust}, let us introduce some notation.

\begin{notation} If $x,y$ are points in the circle, the interval notation $(x,y)$ denotes the smallest of the two intervals in the complement of the two points (when there is a smallest one). Similarly one can define $[x,y)$, $(x,y]$ and $[x,y]$.
\end{notation}

\begin{notation} If $\epsilon>0$ is given, and $\epsilon\mapsto\delta(\epsilon)>0$, $\epsilon\mapsto g(\epsilon)>0$ are real functions, we use the following notations:
\begin{itemize}
\item $\delta = O(g(\epsilon))$ if there is a constant $C$ and $\epsilon_0>0$ such that $\delta(\epsilon)<Cg(\epsilon)$ whenever $0<\epsilon<\epsilon_0$.
\item $\delta \approx g(\epsilon)$ if $g(\epsilon)/\delta(\epsilon) \to 1$ as $\epsilon\to 0^+$.
\end{itemize}
\end{notation}
\subsection*{Proof of Theorem \ref{crit-1}}
We will prove that $f$ has no turning point. This is enough to complete the proof, because it implies that $f$ is a local homeomorphism, and from Theorem \ref{th:localhomeo} one concludes that $f$ is topologically conjugate to a linear expanding map as required.

Let $c_1,\dots, c_k$ be the turning points, and fix $\gamma>0$ such that for each $i$, the interval $J_i = (c_i-\gamma, c_i+\gamma)$ is such that $f(J_i) = (a_i,f(c_i)]$ for some $a_i$, and $f$ is injective in $(c_i-\gamma, c_i]$ and $[c_i,c_i+\gamma)$. We may also assume that the intervals $J_1,\dots, J_k$ are pairwise disjoint

Fix $\epsilon>0$ and $1\leq i\leq k$, and let $z_i\in J_i$ be such that $\dist(z_i,c_i)=\epsilon$. 
Then $\dist(f(z),f(c_i)) = O(\epsilon^2)$, because $f$ is $C^2$ and $c_i$ is a turning point. The pseudo orbit 
$$\{z_i,   f(c_i), f^2(c_i),\dots, f^{j}(c_i)\dots\},$$
has a single ``jump'' of length $O(\epsilon^2)$, and therefore it is $\delta$-shadowed by the orbit of a point $x_i$ for some $\delta = O(\epsilon^{2\alpha})$. Thus $\dist(x_i, z_i)\leq \delta$ and $\dist(f^{j}(x_i), f^j(c_i)) \leq \delta$ for all $j\geq 1$.
 If $\epsilon$ is small enough then $x_i\in J_i$ and there is a point $x_i'\in J_i$ such that $f(x_i')=f(x_i)$ and $c\in (x_i,x_i')$. It is easy to see that $\dist(x_i', c_i)\approx \dist(x_i,c_i)) \geq \epsilon-\delta$. In particular, provided that $\epsilon$ is small enough,  we have that $(c_i-\epsilon/2, c_i+\epsilon/2)\subset I_i := (x_i, x_i')$ (note that $2\alpha>1$, so that $\delta<\epsilon/2$ if $\epsilon$ is small). Also observe that $f(I_i) = (f(x_i), f(c_i)]$.

\begin{claim} There is $n_i\in \{1,\dots, k\}$ and $j_i\geq 0$ such that $c_{n_i}\in f^{j_i}(I_i)$ and $\diam(f^{j_i}(I_i)) \leq \delta$.
\end{claim}

\begin{proof}  Let us first show that $f^{j}(I_i)$ contains a turning point for some $j\geq 1$: By transitivity, there is $m>0$ such that $f^m(I_i)\cap I_i\neq \emptyset$. The set $L=\bigcup_{n=1}^{\infty} f^{nm}(I_i)$ is connected, so it is either an interval or the whole circle. Moreover, $f^m(L)\subset L$. Suppose for contradiction that $f^j(I_i)$ does not contain a turning point, for any $j\geq 1$. Then neither does $L$, and since there is at least one turning point, it follows that $L$ is not the whole circle. Thus $L$ is an interval such that $f^m(L)\subset L$, and since $f$ has no turning point in $L$ it follows that $f^m|_L$ is injective. This implies that the $\omega$-limit of any point in $L$ by $f^m$ is a semi-attracting fixed point for $f^k$, which contradicts the transitivity of $f$. 

Now let $j_i$ be the first positive integer such that $f^{j_i}(I_i)$ contains a turning point, and let $c_{n_i}$ be such turning point. Since for $1\leq j< j_i$ there is no turning point in $f^j(I_i)$, it follows that $f|_{f^j(I_i)}$ is injective, and so $f^{j_i-1}|_{f(I_i)}$ is injective. This implies that $f^{j_i}(I_i) = (f^{j_i}(x_i), f^{j_i}(c_i))$, and so  $\diam(f^{j_i}(I_i)) = \dist(f^{j_i}(x_i), f^{j_i}(c_i)) \leq \delta$, as claimed. 
\end{proof}

To complete the proof of the theorem, let us use the notation $A\Subset B$ to mean that $\ol{A}\subset B$. Note that if $\epsilon$ is chosen small so that $\delta<\epsilon/2$, we have that $f^{j_i}(I_{i})\Subset I_{n_i}$ because  $I_{n_i}$ contains $(c_{n_i}-\epsilon/2, c_{n_i}+\epsilon/2)$ and $\diam(f^{j_i}(I_{i}))<\epsilon/2$. 
The map $\tau \colon \{1,\dots, k\} \to \{1,\dots, k\}$ defined by $\tau(i) = n_i$ has a periodic orbit because the space is finite, so that there is a sequence $i_1, i_2,\dots, i_m=i_1$ such that $f^{j_{i_r}}(I_{i_r}) \Subset I_{n_{i_{r+1}}}$ for $1\leq r< m$. Letting $N=j_{i_1}+j_{i_2}+\cdots+j_{i_{m-1}}$, we conclude that $f^{N}(I_{i_1}) \Subset I_{i_1}$. This contradicts the transitivity of $f$, completing the proof.

\qed

\subsection*{Proof of Theorem \ref{crit-robust}}
From Theorem \ref{crit-1} follows that $f$ has is a local homeomorphism conjugate to a linear expanding map. Let $\pi\colon \R\to \mathbb{S}^1= \R/\Z$ be the universal covering, and $F\colon \R\to \R$ a lift of $f$. Write $F_t(x)=F(x)+t$, and let $f_t\colon \mathbb{S}^1\to \mathbb{S}^1$ be the map lifted by $F_t$. Since $f$ has no turning points and preserves orientation, $F$ is increasing, and the same is true for $F_t$. Since $f$ is $C^r$-robustly transitive, there is $\gamma>0$ such that $f_t$ is transitive for all $t\in (-\gamma,\gamma)$

\setcounter{claim}{0}
\begin{claim}\label{claim:rob1} For every open interval $I\subset \mathbb{S}^1$ and $x\in \mathbb{S}^1$ there is $n>0$ and $y\in I$ such that $f^n(y)=x$.
\end{claim}
\begin{proof} We need to show that $\bigcup_{n\geq 0}f^n(I)=\mathbb{S}^1$. Suppose not. Since $f$ is transitive, there is $k\in \N$ such that $f^k(I)\cap I\neq \emptyset$. Let $L=\bigcup_{n\geq 0} f^{nk}(I)$. Then $L$ is a connected set, so it is either an interval or the whole circle, and $f^k(L)\subset L$. Suppose $L$ is an interval. Then $f^k|_L$ is injective, because it has no turning points. This implies that there is a semi-attracting fixed point for $f^k$ in $L$, contradicting the transitivity of $f$. Thus $L=\mathbb{S}^1$, and this proves the claim.
\end{proof}

\begin{claim}\label{claim:rob2} For each $z\in \R$ and $t>0$, $F_t^n(z) \geq F^n(z) + t$.
\end{claim}
\begin{proof} By induction: the case $n=1$ is trivial, and assuming $F_t^n(z)\geq F^n(z)+t$, we have $F_t^{n+1}(z) = F(F_t^n(z))+t \geq F(F^n(z)+t)+t \geq F^{n+1}(z)+t$ due to the monotonicity of $f$. 
\end{proof}

\begin{claim}\label{claim:rob3} For each $\epsilon>0$ and $x\in \mathbb{S}^1$ there is $n\in \N$ and $0< t < \epsilon$ such that $f_t^n(x)=x$. 
\end{claim}
\begin{proof} 
Let $\til{x}\in \R$ be such that $\pi(\til{x})=x$, and define $\til{I}=(F(\til{x}), F(\til{x})+\epsilon)$. Claim \ref{claim:rob1} implies that there is $y\in I=\pi(\til{I})$ and $n\in \N$ such that $f^n(y) = x$. This means that if $\til{y}$ is the point in $\til{I}$ such that $\pi(\til{y})=y$, then there is $m\in \Z$ such that $F^n(\til{y}) = \til{x}+m$. Observe that we can write $\til{y}=F(\til{x})+s=F_s(\til{x})$ for some $s$ with $0< s< \epsilon$. By the previous claim, $F_s^{n+1}(\til{x}) = F_s^n(\til{y})\geq F^n(\til{y})+s> \til{x}+m$. On the other hand, since $F(\til{x})<\til{y}$ and $F^n$ is increasing, we have $F_0^{n+1}(\til{x}) = F^n(F(\til{x})) < F^n(\til{y}) = \til{x}+m$. Thus, by continuity, there is $0<t<s$ such that $F_t^{n+1}(\til{x}) =\til{x}+m$, so that $f_t^{n+1}(x) = x$ as required
\end{proof}

\begin{claim} $f$ has no critical points.
\end{claim}
\begin{proof} Suppose $c$ is a critical point. Claim \ref{claim:rob3} implies that there exist arbitrarily small choices of $t>0$ such that $f_t^n(c)=c$ for some $n$. But $c$ is also a critical point for $f_t^n$, and so it is an attracting fixed point for $f_t^n$, contradicting the transitivity of $f_t$.
\end{proof}	

Theorem A in \cite{M} implies that if $f$ is a transitive $C^2$ endomorphism without critical points, then one of the following hold:
\begin{enumerate}
\item[(1)] $f$ is topologically conjugate to a rotation;
\item[(2)] $f$ has a non-hyperbolic periodic point;
\item[(3)] $f$ is an expanding map.
\end{enumerate}

We can rule out case (1), since $f$ is not a homeomorphism. In fact, if $f$ is a homeomorphism, Claim \ref{claim:rob3} implies that there exist arbitrarily small values of $t$ such that $f_t$ has periodic points, and being $f_t$ a homeomorphism, it follows that $f_t$ is non-transitive, a contradiction.

To finish the proof of the theorem, we have to rule out case (2) above, \ie we need to show that all periodic points of $f$ are hyperbolic. Suppose $p$ is a non-hyperbolic fixed point of $f$, and let $k\in \N$ be the least period, so that $f^k(p)=p$ and $(f^k)'(p)=1$ (because $f^k$ is increasing). Let $I  = (p-\epsilon, p+\epsilon)$ with $\epsilon$ so small that $f^i(p)\notin I$ for $1<i<k$, and choose a $C^\infty$ map $h\colon \mathbb{S}^1\to \mathbb{S}^1$ which is $C^r$-close to the identity, such that $h(x) = x$ if $x\notin I$, $h(p)=p$ and $0<h'(p)<1$. Then $hf$ is $C^r$-close to $f$, and in particular it is transitive. But $0<((hf)^k)'(p)<1$, so $p$ is a periodic sink for $hf$, contradicting the transitivity. This proves that all periodic points of $f$ are hyperbolic, completing the proof.
\qed

\section*{Appendix: Creation of crooked horseshoes near the identity}
\label{sec:crooked}

\begin{figure}[ht!]
\begin{center}
\includegraphics[height=7cm]{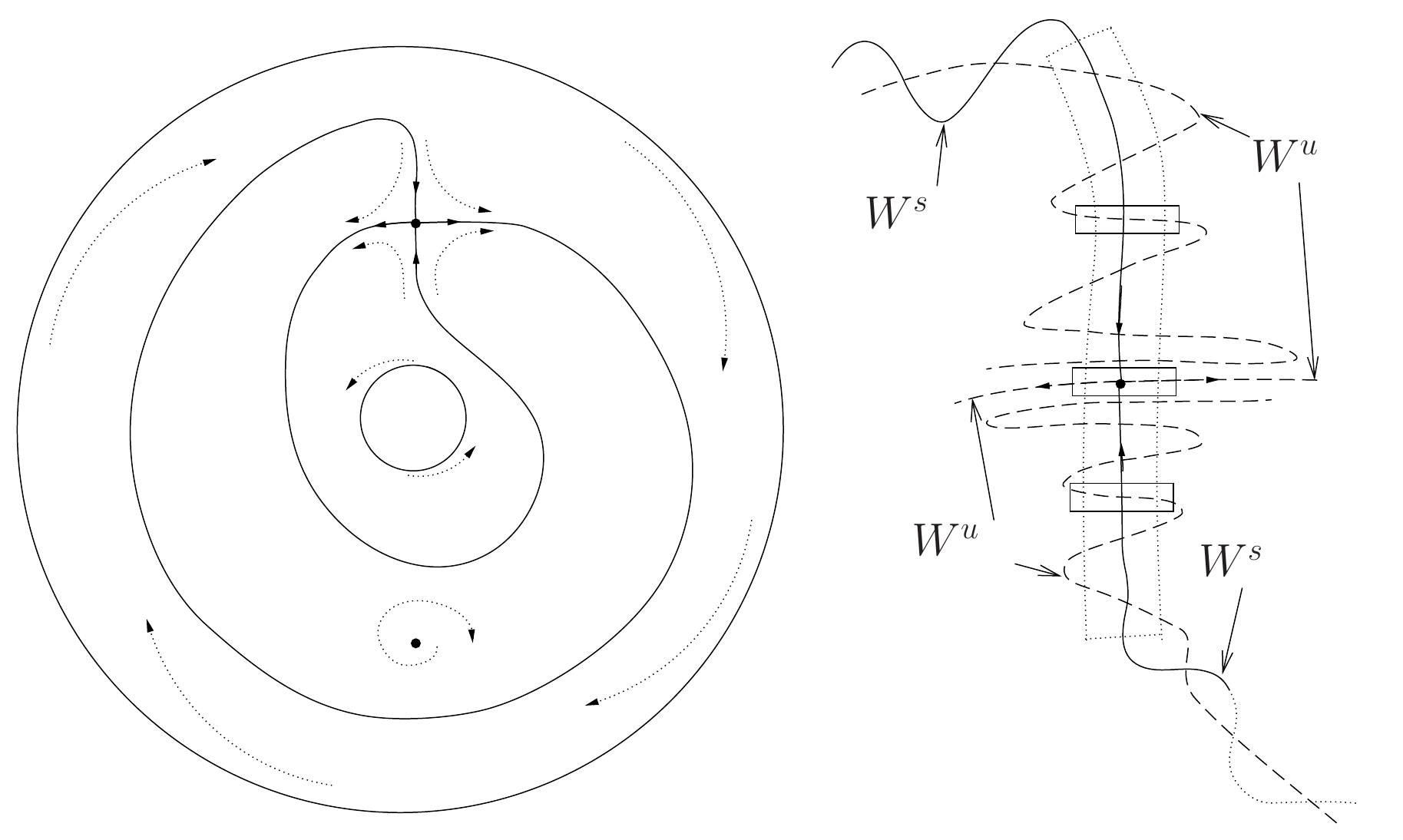}
\caption{Creating a crooked horseshoe close to the identity from a flow}
\label{fig:flow}
\end{center}
\end{figure}

Given $m>0$, we will construct a $C^r$ diffeomorphism satisfying the properties required in Proposition \ref{pro:flow}.
For that, we start with a vector field $X$ exhibiting a double loop between the stable and unstable manifold of a hyperbolic singularity, as in figure \ref{fig:flow} (see details in subsection \ref{initial}). Moreover, the loop is an attracting set. Then, for each $m$, we take the time-$\frac{1}{m}$ map $\Phi_{\frac{1}{m}}$ of the flow associated to $X$. The fact that this map is $\frac{C_r}{m}$-close in the $C^r-$topology to the identity map where $C_r$ is a positive constant independent of $m$ follows from Lemma \ref{lem:tvm} below. Then, the map is perturbed into a $C^r$ diffeomorphism $f$ breaking an  homoclinic connection between the stable and unstable manifold of the fixed point associated to each loop (see figure \ref{fig:flow}). It is proved that this diffeomorphism is Axiom A with strong transversality. To prove that, we adapt to the present context the strategy developed in \cite{newhouse-palis}. Note that the estimate on the distance to the identity is preserved if the perturbation is $C^r$-small. The resulting map has the properties mentioned in Proposition \ref{pro:flow} (one can verify the presence of a crooked horseshoe for $f^m$ using standard arguments). We devote the rest of this section to obtaining the required perturbations.

We state an elementary fact that was used in the previous description.
\begin{lemma}\label{lem:tvm} Let $X$ be a $C^{r+1}$ vector field on the compact manifold $M$, and $\phi\colon M\times [0,1]\to M$ the associated flow. Then $\Phi_t = \Phi(\cdot, t)$ is such that $d_{C^r}(\Phi_t, \id) < C_r t$ for some constant $C_r$ independent of $t$.
\end{lemma}
\begin{proof}[Sketch of the proof] We prove it locally; the global version is obtained by standard arguments.
Assume the flow is defined on a neighborhood of $\ol{U}$ for some bounded open set $U\subset \R^n$. Let $$C_r=\max_{0\leq k\leq r}  \norm{D_x^k X}$$
were $D^k_xX(x) \colon (\R^n)^k\to \R^n$ is the $k$-th derivative of $X$ and $\norm{D_x^kX}$ is the supreme of $\norm{D_x^kX(x)}$ for $x$ in a neighborhood of $\ol{U}$.

By the mean value inequality, if $t$ is small enough,
\begin{align*}
\norm{D_x^k\Phi_t-D_x^k\id} = \norm{D_x^k\Phi_t - D_x^k\Phi_0} \leq \sup_{0\leq s\leq t} \norm{\frac{d}{ds}D_x^k\Phi_s} \\= \sup_{0\leq s\leq t} \norm{D_x^k\frac{d}{ds} \Phi_s} \leq \norm{D_x^k X}t \leq C_r t.
\end{align*}
\end{proof}

This  section is organized as follows: first we introduce the vector field $X$; later, we consider perturbations of the map $\Phi_\frac{1}{m}$, by embedding the map in a one-parameter family; finally, we prove that for certain parameters the map is Axiom A.

To simplify the proof, we will assume that the double connection of our flow is as in figure \ref{vector-field} on the sphere, which does not make a difference since after compactifying by collapsing boundary components of the annulus in figure \ref{fig:flow} to fixed repelling points, we are in the same setting (in figure \ref{vector-field}, the sources $R_1$ and $R_3$ correspond to the collapsed boundary components, while $R_2$ is the source inside the loop of figure $\ref{fig:flow}$), and the perturbations that we are going to use are supported in the complement of a neighborhood of $R_1$ and $R_2$.

\subsection{The vector field $X$} \label{initial} 

\begin{figure}[ht]
\begin{center}
    \includegraphics[height=7cm]{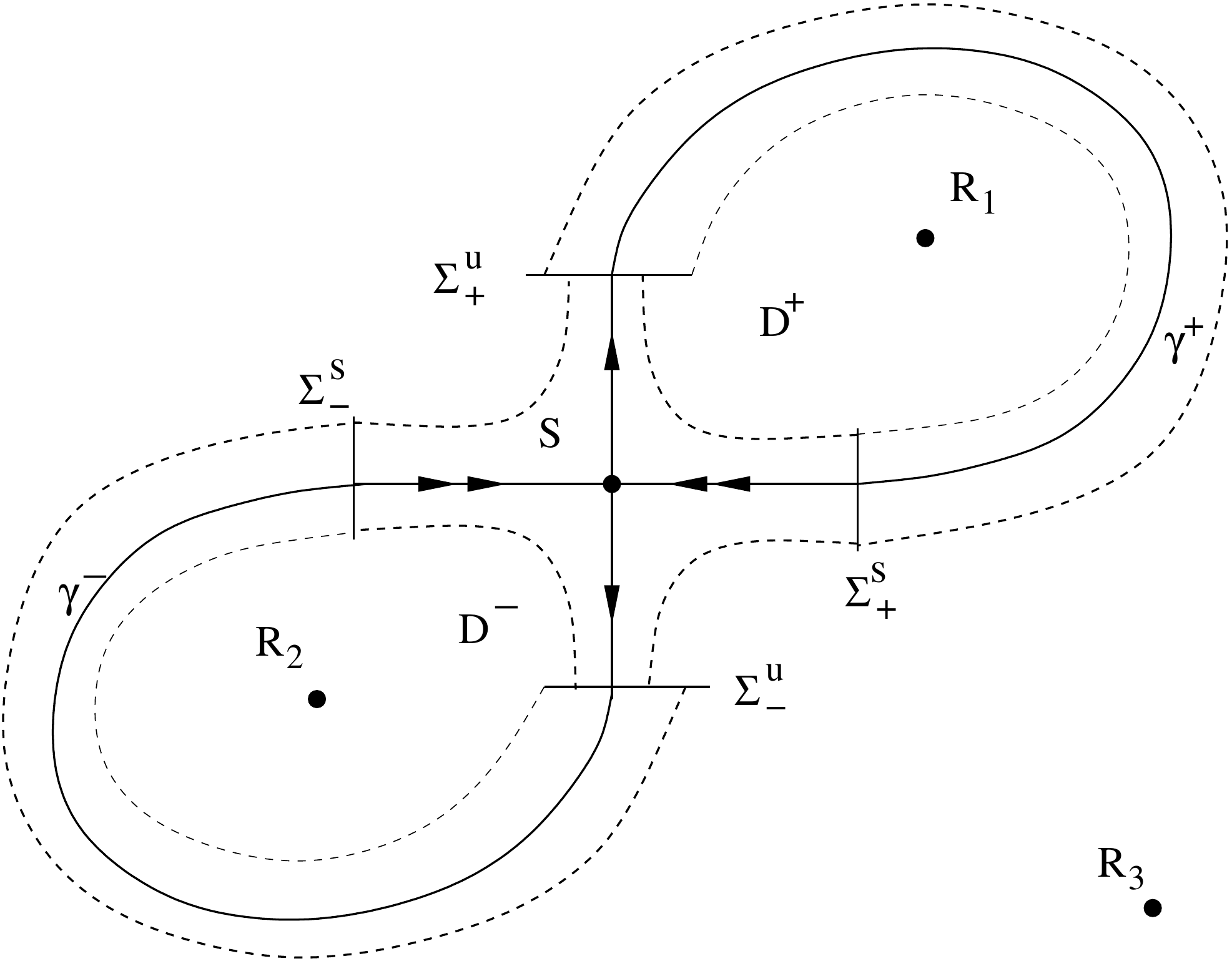}
\end{center}
\caption{Vector field $X$}
\label{vector-field}
\end{figure}

Let us consider a vector field $X$ defined on the two dimensional 
sphere such that (after composing with the stereographic projectiion)  in the disk $D=[-2,2]\times [-2,2]$ the following holds:
\begin{enumerate}
\item It is symmetric respect to $(0,0)$, i.e. $X(-p)=-X(p).$
\item $S=(0,0)$ is a hyperbolic saddle singularity such that:
\begin{enumerate}
\item inside $[-\frac{1}{2}, \frac{1}{2}]\times [-\frac{1}{2}, \frac{1}{2}]$ the vector field is the linear one given by $X(x,y)=(\log(\la) x, \log(\s) y)$ with $0<\la<1<\s$, $\la\s<1$, $\la=\s^{-\ga},$ with $ \ga >\max\{3r, 6\}$  where $r$ is the smoothness required; in particular,  $[-\frac{1}{2}, \frac{1}{2}]\times \{0\}$ is contained  in the  local  stable manifold of $S$ and $\{0\}\times [-\frac{1}{2}, \frac{1}{2}]$ is contained  in the  local  unstable manifold of $S.$
\item the stable and the unstable  manifold of $(0,0)$ are contained in $D$;
\item the stable manifold and the unstable one form a loop $\ga$ contained in $[0,2]\times [0,2]\cup [-2,0]\times [-2,0]$;  let us denote by $\ga^+$ the one in  $[0,2]\times [0,2]$ and $\ga^-$ the one in  $[-2,0]\times [-2,0].$
\end{enumerate}

\item $R_1=(1,1)$ is a hyperbolic repelling singularity and is contained in the region $D^+$ bounded by the loop $\ga^+$ (by symmetry $R_2=(-1,-1)$ is a hyperbolic repelling singularity and is contained in the region $D^-$ bounded by the loop $\ga^-$).

\item $R_3=(2,-2)$ is a hyperbolic repelling singularity.

\item Let $\Si^u_+$ be the transversal section $[-\frac{1}{3}, \frac{1}{3}]\times \{\frac{1}{2}\}$ and  
 $\Si^s_+$ be the transversal section $\{\frac{1}{3}\}\times [-\frac{1}{3}, \frac{1}{2}]$, and $P^+:\Si^u_+\to \Si^s_+$ be the induced  map by the flow, it is assumed that $P^+(x)=x.$ By symmetry there is also an induced  map $P^-$ defined from $\Si^u_-=[-\frac{1}{3}, \frac{1}{3}]\times \{-\frac{1}{2}\}$ to 
 $\Si^s_-=\{-\frac{1}{2}\}\times [-\frac{1}{3}, \frac{1}{2}]$ and $P^-(x)=x$. 

\item The only singularities in $D$ are $S, R_1, R_2, R_3.$ 

\end{enumerate}

\begin{remark}\label{preserve}  From the choice of the eigenvalues of the singularity, observe that the induced map by the flow $L^+$ from  $\Si^s_+\setminus\{y=0\}$ to $\Si^u_+\cup \Si^u_-$ is a strong contraction, provided that $\Si^s_+$ is a small neighborhood of $\Si^s_+\cap\{y=0\}$. In the same way it follows that  the induced map  by the flow $L^-$ from  $\Si^-_+\setminus\{y=0\}$ to $\Si^u_+\cup \Si^u_-$ is also a contraction. Therefore, the return map from $\Si^s_+\cup \Si^s_-\setminus\{y=0\}$ to itself is a contraction. From that, it follows that the loop $\ga^+\cup \ga^-$ is an attracting loop.

\end{remark}

\begin{lemma} With the assumptions above, the vector field $X$ can be built in such a way that, 
\begin{enumerate}
 \item the repelling basin of $R_1$ is given by $D^+$ and the repelling basin of $R_2$ is   $D^-;$
\item the repelling basin  of $R_3$ in $D$  is the complement of  $D^+\cup D^-\cup \ga^+\cup \ga^-;$
\item the non-wandering set of $X$  is $S,R_1, R_2, \ga^+,\ga^-, R_3.$ 
\end{enumerate}

\end{lemma}
\begin{proof} Let  $\be^+$ be a  simple closed curve  inside  $D^+$ and close to $\ga^+$ and let $T^+$ be the  annulus bounded by $\ga^+$ and $\be^+$.  By the property on the return map $R$ inside $D^+$ and that the saddle $S$ is dissipative  (see Remark \ref{preserve}), it follows that $\Phi_t^X(T^+)\subset T^+$ for large  $t>0$. This allows to build $X$ in such a way that the first item holds.

With a similar argument, observe that if it is taken any  closed curve  $\al$ outside   $D^+\cup D^-$ and close to $\ga^+\cup \ga^-$, and $T$ is an annulus bounded by $\ga^+\cup \ga^-$ and $\al$, by the property on the return map and that the saddle $S$ is also dissipative, it follows that $T\subset \Phi_t^X(T)$ for large $t>0$. This allows to build $X$ in such a way that the second item holds.

The last item in the thesis of the lemma is a consequence of the first and second item.

\end{proof}

\subsection{The flow  $\Phi_\frac{1}{m}$} Now, given $m$, we take $f=\Phi_{\frac{1}{m}}$, the $\frac{1}{m}-$time  map of the flow associated to $X.$ For the sake of simplicity, we can assume that there exist $b>0$ with 
$(0,b)\in W^u_{loc}(S)$,  $a>0$ with  $(a, 0)\in W^s_{loc}(S)$, and $k_m>0$ such that $f^{k_m}(0,b)=(a,0)$. The iterate $k_m$ depends on $m$ (and in fact $k_m\to +\infty$) but from now on, for simplicity, we assume that $k_m$ is equal to $2$.  Let $L^u_+=[(0,b),f(0,b)]$ be a  fundamental domain  inside the local unstable manifold of $S$ and  $L^s_+=[f(a,0),(a,0)]$ be a  fundamental domain inside the local stable manifold of $S$.  Let us take  $B^u_+=[-\e,\e]\times L^u_+$ and  $B^s_+=L^s_+ \times[-\e,\e]$ with $\e$ small, and observe that  reparameterizing the time flow,    we can assume that  $$f^2(x,y)= (f^u(y), x),$$  where $f^2=f\circ f$,  $f^u:L^u_+\to L^s_+$ is a one-dimensional diffeomorphism such  that 
$f^u(b)=a,$ and ${f^u}' <c<0.$
By symmetry, the same holds in the neighborhood $B^u_-=L^u_-\times [-\e, \e]=-L^u_+\times [-\e, \e]$, $B^s_-=-L^s_+\times [-\e,\e]$ and in particular,  
$f^2(0,-b)=(-a,0)$. Of course, the fundamental domains chosen depend on $m$, more precisely, as $m$ is larger, the fundamental domains gets smaller (recall that $\Phi_{\frac{1}{m}}$ converge to the identity map).

\subsection{Perturbations of $\Phi_\frac{1}{m}$}\label{pert initial}  First we embed the map $f=\Phi_\frac{1}{m}$ in a one-parameter family $\{f_{t}\}_{ t\geq 0}$ where $f_{0}=f$.  Now for each $ t>0$ small, we get a diffeomorphism  $f_{t}$ $C^r$  close to $f$. 
Moreover, we can get $f_{t}$ satisfying  the following properties (details about the construction of $f_{t}$ are in subsection \ref{about}):

\begin{enumerate}
 \item  The map $f_t$ is symmetric respect to $(0,0)$, i.e. $f_t(-p)=-f_t(p).$

 \item If $t$ is small, $S, R_1, R_2$, and $R_3$ are hyperbolic fixed points.

\item For any $t$, the dynamics in $[-\frac{1}{2}, \frac{1}{2}]\times [-\frac{1}{2}, \frac{1}{2}]$ is given by $f_t(x,y)=(\la x, \s y)$, where $\la:=\la^{\frac{1}{m}}, \s:=\s^{\frac{1}{m}}.$ Observe that it is verified that $\la=\s^{-\ga}$, the saddle fixed point $S=(0,0)$ is dissipative, and the local unstable manifold is contained in the $y$-axis and the local stable in the $x$-axis.

\item Provided the neighborhoods   $B^u_+=[-\e,\e]\times L^u_+$ and $B^s_+=L^s_+ \times[-\e,\e]$,  we assume that (see figure \ref{pert})
$$f^2_t(x,y)= (f^u(y), x+f^s_{t}(y))$$ where $f^s_{t}: L^u_+\to \R$ is a $C^r$   function  verifying that
 \begin{enumerate} 
\item $f^s_{0}=0$,

\item there exists a unique point $b''$ with  $b< {b}''< {b}'$ such that $f^s_{ t}({b}')=f^s_t({b}'')=f^s_t(b)=0$ (where ${b}'$ is a point such that $(0, {b}')=f(0, b)$),

\item for any $y\in (b, {b}'')$ follows that $f^s_t(y)>0, {f^s}^{''}_t(y)<0$ and  for any $y\in ({b}'', {b}')$ follows that $f^s_t(y)<0, {f^s}^{''}_t(y)>0,$ 

\item for any $t$ the map $f^s_t$ has only two critical points $c_1\in (b, {b}'')$ and $c_2\in ({b}'', {b}')$ such that $f^s_t(c_1)=t, f^s_t(c_2)=-t(1+\de)$ where $\de+1=\frac{c_2}{c_1}$ and moreover the critical points are not degenerated. In particular, $f^2(0,c_1)= (c_1',t),  f^2(0,c_2)= (c_2',-t(1+\de)).$
\end{enumerate}

\item The function $f_t$ coincide with $f_{0}$ outside a neighborhood of size $t^{\frac{1}{r}}$ of $L^u\cup - L^u$ for $r$ large.  In particular, the map $f_t$ can be built in such a way that is $C^r$ close to $f_{0}$.

\end{enumerate}

\begin{figure}[subsection]
\begin{center}
    \includegraphics[height=3cm, width=12cm]{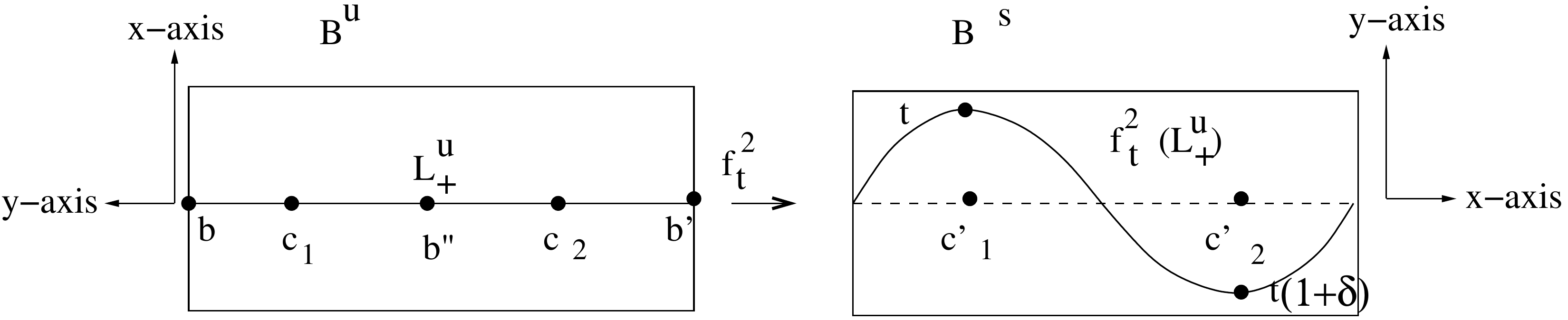}
\end{center}
\caption{Map on fundamental domains.}
\label{pert}
\end{figure}
Observe that by symmetry, provided  the neighborhood   $B^u_-=[-\e,\e]\times -L^u$ and $B^s_-=-L^s \times[-\e,\e]$,  it follows  that 
$$f^2_t(x,y)= (f^u(y), x-f^s_t(-y))$$ where $f^s_t$ is the map defined before and so $-f^s_t\circ -Id:-L^u\to \R$ is a $C^r$  family of functions  verifying symmetric similar properties to the one listed above.

\begin{remark}
Without loss of generality, we can assume that the critical points of $f^s_t$ are quadratic:
\begin{enumerate}
 \item $f^2_t(x,y)=(y, -(y-c_1)^2+t+x)$  nearby the point $(0,c_1)$, 
\item $f^2_t(x,y)=(y, (y-c_2)^2-t(1+\de)+x)$ nearby the point $(0,c_2)$,
\item $f^2_t$ it is defined by symmetry nearby the points $(0, -c_1)$ and $(0, -c_2)$.
\end{enumerate}
\end{remark}

\begin{figure}[subsection]
\begin{center}
   \includegraphics[width=7cm, height=7cm]{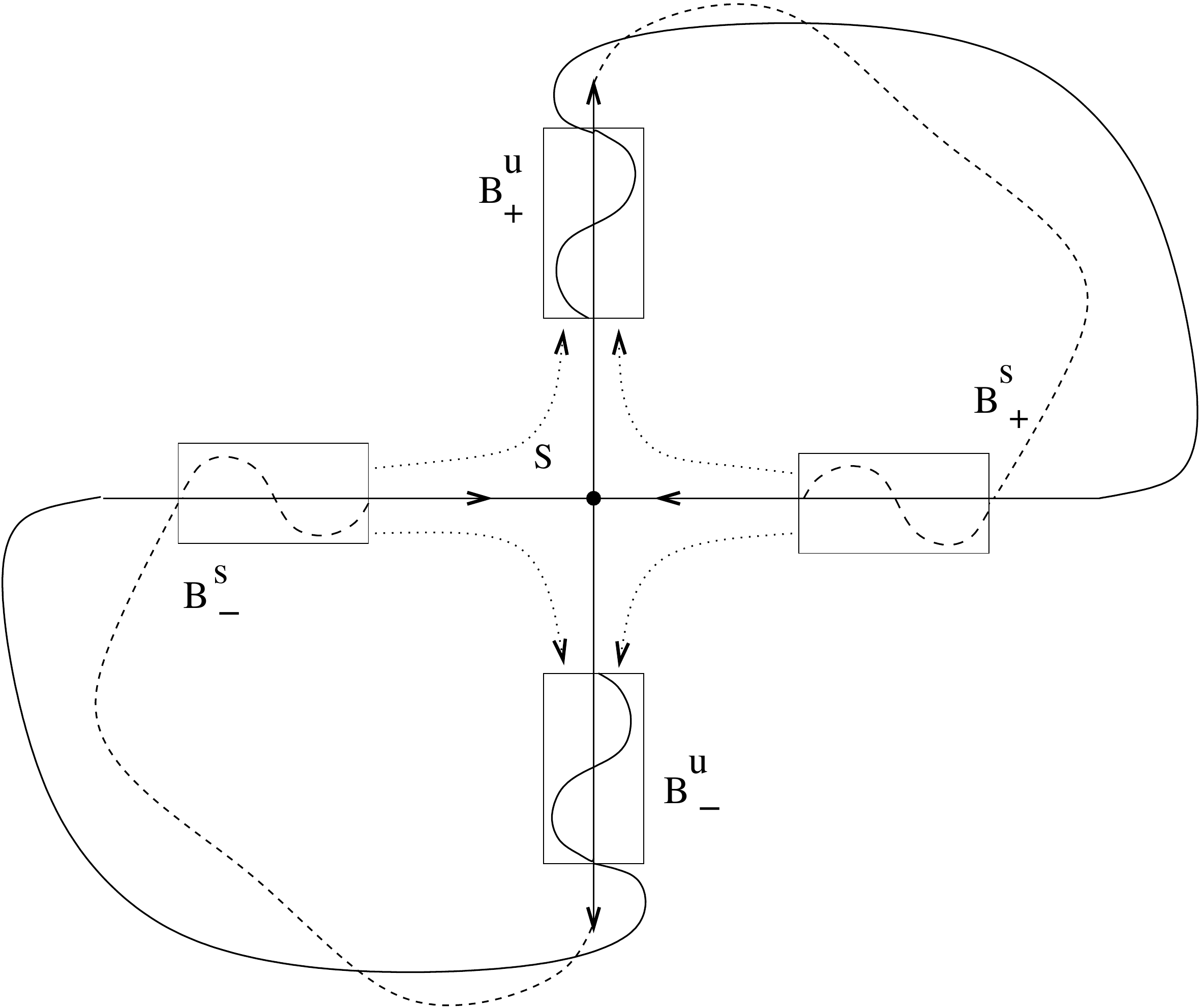}
\end{center}
\caption{$f_t$.}
\label{map}
\end{figure}

\begin{remark} \label{viz} There exists an arbitrary  small neighborhood $W$ of $\ga^+\cup \ga^-$ such that provided $ t$ small then 
 any point $x\in W^c$ belong  either to the basin of repelling of the fixed points  $R_1$ $R_2$ or $R_3.$ Moreover, $f_t(W)\subset W$ for any $t$ small.
\end{remark}

\subsection{Axiom A for certain parameters} 
In \cite{newhouse-palis} it was proved that given a $C^2$ one-parameter family displaying a $\Omega-$explosion, and assuming that before the explosion the nonwandering set is given by hyperbolic periodic orbits and a single tangent  homoclinic orbit, then there exist parameters arbitrarily close to $0$ such that the diffeomorphism is a non trivial Axiom A. These results cannot be applied straightforwardly in our present context; however, the proof can be adapted to conclude the following proposition that allows to prove Proposition \ref{pro:flow}.

\begin{proposition}\label{hyp} For any $n$ sufficiently large, if $t=\s^{-n}c_1$  then  $f_t$ is Axiom A with strong transversality condition. Moreover  $\Omega(f_t)$ is formed by
\begin{enumerate}
\item three hyperbolic attracting periodic orbits $p_+, p_-, q$; $p_+$ contained in $D^+$, $p_-$ in $D^-$ and $q$ in the complement of $D^-\cup D^+;$ 
\item the repelling fixed points $R_1, R_2, R_3$;
\item  a finite number of transitive hyperbolic compact invariant sets contained in the complement of the basin of attraction and repelling of $R_1, R_2, R_3,$ $p_+, p_-, q$   and containing   the homoclinic class  of $S$.
\end{enumerate}
 \end{proposition}
The strategy to prove Proposition \ref{hyp} consists in the following steps:
\begin{enumerate}
 \item We choose the  parameter $t$ (arbitrary small) such that all   the critical points of $f^s_t$ belong to the basin of attraction of some    attracting  periodic points (see Lemmas \ref{pozo1},  \ref{pozo2} and \ref{pozo3}).
\item The nonwandering set of $f_t$ is contained in the union of the maximal invariat set of $f_t$ in $W$ (i.e., $\cap_{n\in \N}f^n_t(W)$)  and the repelling fixed points $R_1,R_2, R_3$ (it follows from remark \ref{viz}).
\item The maximal invariant  set of $f_t$ in $W$ intersected with  the boxes $B^s_+\cup B^s_-$, is contained in a narrow strip  along the curves $f^2_t(L^u_+\cup L^u_-)$ (see Lemma \ref{nonwander}).
\item It is proved that the part of the nonwandering set of $f_t$ contained in $W$ and in the complement of the attracting periodic points listed in the first item has a dominated splitting (see Lemma \ref{hyplemma}). 
\item Using Theorem B in \cite{pujals-sambarino} it is concluded that $f_t$ is Axiom A and satisfies the strong transversality condition.
\end{enumerate}

To prove item $1$ in the above described strategy,  in the next lemma, we  prove that for  $t=\s^{-n}c_1$ small enough, the critical points of $f^s_t$ belong to the basin of attraction of an  attracting  periodic point. For that value of $t$ it also follows that $(c_2', -t)$ and $(-c_2', t)$ belong  to the basin of attraction of a periodic point in $[D^+\cup D^-]^c$ (see Lemma \ref{pozo3}). 

\begin{lemma}\label{pozo1}  For any $n$ sufficiently large, if $t=\s^{-n}c_1$ then there exists a hyperbolic attracting periodic point $p_+$ contained in $D^+$ such that, for $2<\al<3$, the following hold:
\begin{enumerate}
\item  $B^+_\al(c_1)=\{(x,y): 0\leq x\leq t^{\al},  |y-c_1| < 2^\frac{1}{2} t^\frac{\al}{2} \}$ is contained in the local basin of attraction of $p_+,$

\item   $B^-_\al(c_1')=\{(x,y): |x- c_1'|< 2^\frac{1}{2} t^\frac{\al}{2}, | y-t|< t^{\al}\}$  is contained in the local basin of attraction of $f^2_t(p_+)$ and in particular,  $(c_1', t)$ is in the basin of $p_+$.
\end{enumerate}
 
\end{lemma} 

\begin{proof} First it is proven that  for $t=\s^{-n}c_1$ small  
\begin{eqnarray}\label{inc0}
f^{n+2}_t(B^-_{\al}(c_1'))\subset B^-_\al(c_1').
\end{eqnarray}
Observe that for that value of $t$, the $y-$coordinates of $f^n(c_1', t)$ is equal to $c_1.$ This implies that there is an attracting periodic point in the disk $ B^-_\al(c_1')$ with period $n+2$ and latter it is proved that the disk is contained in the basin of attraction.

Observe that $f^2_t(B^+_\al(c_1))\subset  B^-_\al(c_1'),$  and therefore  to prove inclusion  (\ref{inc0})  it is enough to prove that   
\begin{eqnarray}\label{inc2}
 f^n_t(B^-_{\al}(c_1'))\subset B^+_\al(c_1),
\end{eqnarray}
i.e.,  for any $(x, y)\in B^-_\al(c_1')$ we prove  
(i) $ c_1-t^\al< \s^ny\leq c_1+t^\al$ and (ii) $ \la^nx< t^\al$, where $ \la^nx $ and $\s^ny$ are the $x$ and $y$ coordinates of $f^n_t(x,y)$. In fact, on one hand if $(x, y)\in B^-_{\al}(c_1')$ then 
$t-t^\al < y < t+t^\al$ and so $c_1- \s^n t^\al < \s^n y< c_1+ \s^n t^\al$. Since  $t\s^n=c_1$ then $c_1- c_1t^{\al-1} < \s^n y< c_1+ c_1 t^{\al-1}$
but from the fact that $\al>2$ ( which implies that $\frac{\al}{2}< \al-1$) then $c_1t^{\al-1} < t^\al,$ concluding the first inequality. On the other hand, if $(x, y)\in B^-_{\al}(c_1')$ then 
$c_1'-t^{\frac{\al}{2}} < x< c_1'+t^{\frac{\al}{2}}$ and so $0< \la^n y< c_1'\la^n + \la^n t^{\frac{\al}{2}}$. Since  $t\s^n=c_1$ and $\la=\frac{1}{\s^\ga}$ then 
$0< \la^n x< t^\ga$ and recalling  $\al<\ga$ (in fact, $\ga>6$) it follows that $t^\ga<t^\al$ concluding the second inequality and proving inclusion (\ref{inc2}).

To conclude that $ B^-_\al(c_1')$ is inside the basin of attraction, it is shown that for any $z\in  B^-_\al(c_1')$, 
\begin{eqnarray}\label{norm<1}
 ||D_zf^{2n+4}_t||<1.
\end{eqnarray}
Observe that for $z\in  B^+_\al(c_1)$
$$D_zf^2_t=  \left(\begin{array}{cc}
        0 &   {f^u}'\\ 
        1 & \partial_y f^s 
\end{array}\right) $$ with $|\partial_y f^s_t|< 4 t^{\frac{\al}{2}}.$
So, using that $Df^n $ is the diagonal matrix with diagonal $\la^n, \s^n$ then for any $z\in  B^-_\al(c_1')$  
$$ D_zf^{2n+4}_t=\left(\begin{array}{cc}
        (\la\s)^n &   {f^u}'_t(z)\s^n\partial_y f^s_t(z)\\
        0 & (\la\s)^n + \s^{2n}\partial_y f^s_t(z)\partial_y f^s_t(f^{2n+2}_t(z)) 
\end{array}\right) $$ 
and since  $\s^n=\frac{c_1}{t}$ and $f^{2n+2}_t(z)\in B^+_\al(c_1)$ then $$\s^n\partial_y f^s_t(z)< t^{\frac{\al}{2}-1}<<1,\,\,\,|\s^{2n}\partial_y f^s_t(z)\partial_y f^s_t(f^{2n+2}_t(z))|<4t^{2(\frac{\al}{2}-1)}<<1. $$ Since $\la\s<1$ the inequality \ref{norm<1} is proved.
\end{proof}

\begin{figure}[subsection]
\begin{center}
    \includegraphics[height=6cm]{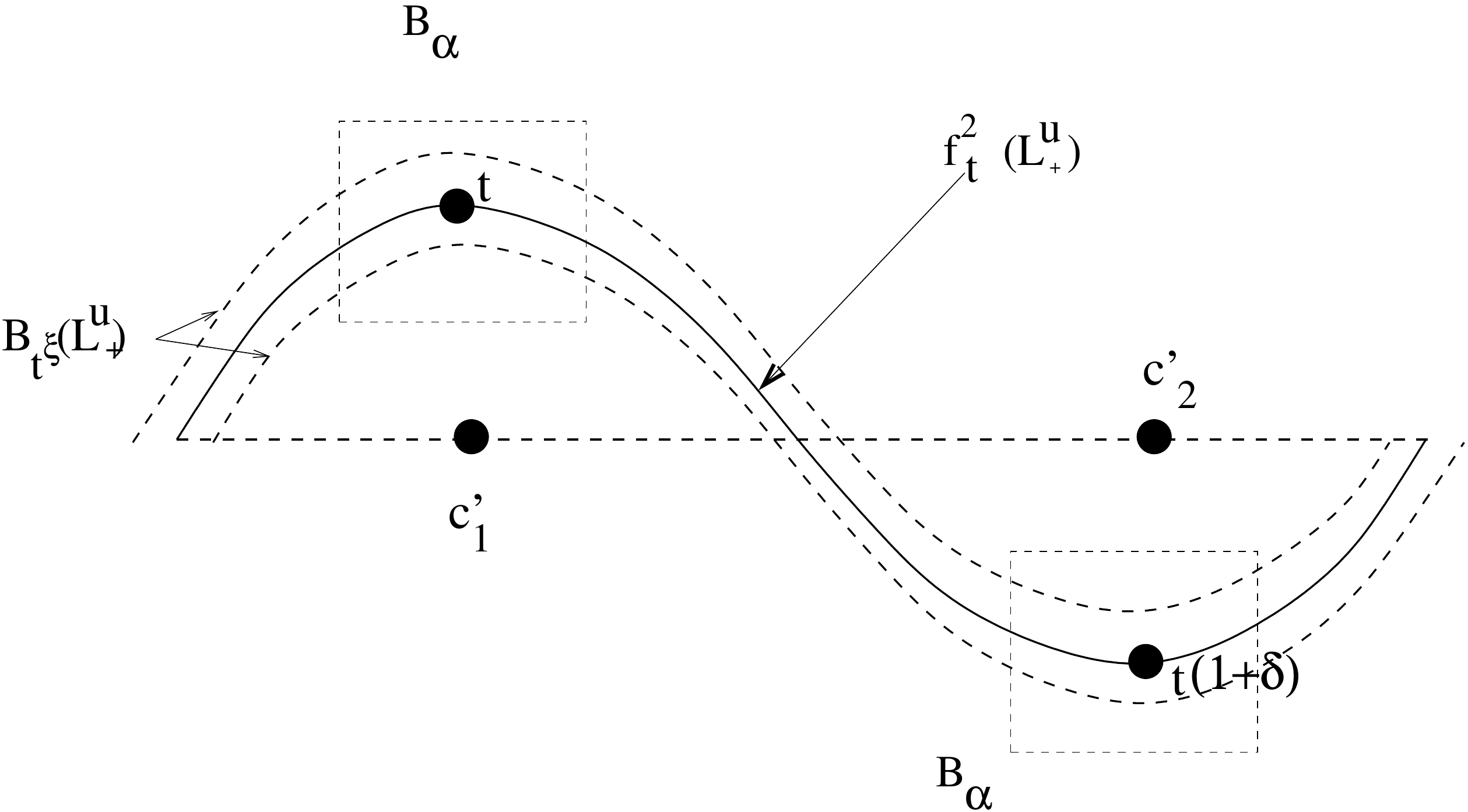}
\end{center}
\caption{Local Basin.}
\label{basin}
\end{figure}

The proof of the next  lemma follows from the symmetric property assumed on $f_t$. It basically states that there is also a sink created on $D^-$.

\begin{lemma}\label{pozo2} For $n$ sufficiently large, if  $t=\s^{-n}c_1$ then  there exists a periodic point $p_-$ contained in $D^-$ such that for $2<\al<3$ it is verified:
\begin{enumerate}
\item the regions $B^+_\al(-c_1)=\{(x,y): -t^{\al}< x < 0,  |y+c_1| < 2^\frac{1}{2} t^\frac{\al}{2} \}$ is contained in the local basin of attraction of $p_-,$

\item   $B^-_\al(-c_1')=\{(x,y): |x+c_1'|< 2^\frac{1}{2} t^\frac{\al}{2}, | y-t|< t^{\al}\}$  is contained in the local basin of attraction of $f^2_t(p_-)$ and in particular,  $(c_1', t)$ is in the basin of $p_+$.

\end{enumerate}
 
\end{lemma} 

The next lemma is about a sink that is created in $D.$ 

\begin{lemma}\label{pozo3}For $n$ sufficiently large, if  $t=\s^{-n}c_1$ then there exists a hyperbolic attracting periodic point $q$ contained in $D=[D^+\cup D^-]^c$ such that for $2<\al<3$ it is verified:
\begin{enumerate}

\item  $B^+_\al(c_2)=\{(x,y): -s^{\al}< x < 0,  |y+c_2| < 2^\frac{1}{2} s^\frac{\al}{2} \}$ ($s=(1+\de)t, \de=\frac{c_2}{c_1}-1$) is contained in the local basin of attraction of $p_+$, 

\item   $B^-_\al(c_2')=\{(x,y): |x+c_2'|< 2^\frac{1}{2} s^\frac{\al}{2}, | y-s|< s^{\al}\}$  is contained in the local basin of attraction of $f^2_t(q)$ and in particular,  $(c_2', -s)$ is in the basin of $q$.

\item the regions $B^+_\al(-c_2)=\{(x,y): 0< x < s^{\al},  |y-c_2| < 2^\frac{1}{2} s^\frac{\al}{2} \}$ is contained in the local basin of attraction of $q,$

\item   $B^-_\al(-c_2')=\{(x,y): |x+c_2'|< 2^\frac{1}{2} s^\frac{\al}{2}, | y-s|< s^{\al}\}$ is contained in the local basin of attraction of $f^2_t(q)$ and in particular,  $(-c_2', s)$ is in the basin of $q$.

\end{enumerate}
 
\end{lemma} 

\begin{proof} Observe that for the construction of $f_t$  it follows that $f^2_t(B^+_\al(c_2))\subset B^-_\al(c_2')$ and $f^2_t(B^+_\al(-c_2))\subset B^-_\al(-c_2')$. On the other hand, since $\s^nt=c_1,$ then $\s^nt(1+\de)= c_2$ (recall that $\de=\frac{c_2}{c_1}-1$).  Repeating the calculation in Lemma \ref{pozo1} and recalling the property of $f_t$ (more precisely, item (4.d) in subsection \ref{pert initial})   follows that  
$f^n_t(B^-_\al(c_2'))\subset B^+_\al(-c_2)$ and $f^n_t(B^-_\al(-c_2'))\subset B^+_\al(c_2)$. Therefore, 
$$f_t^{2n+2}(B^+_\al(c_2))\subset B^+_\al(c_2)$$ and so there is a semiattracting periodic point there. In the same way as in the proof of Lemma \ref{pozo1}, also holds that $$||D_zf^{2n+2}||<1$$ for any $z\in B^+_\al(c_2)$ so the the semiattracting periodic point is a hyperbolic sink such that $B^+_\al(c_2)$ is contained in  its  basin of attraction. A similar argument applies for $B^-_\al(c_2'), B^-_\al(-c_2')$ and $ B^+_\al(-c_2).$

\end{proof}

\begin{corollary}\label{pozocor} There exists $t$ arbitrarily small such that the theses of lemmas \ref{pozo1}, \ref{pozo2} and \ref{pozo3} hold.
\end{corollary}

\begin{figure}[subsection]
\begin{center}
    \includegraphics[height=9cm]{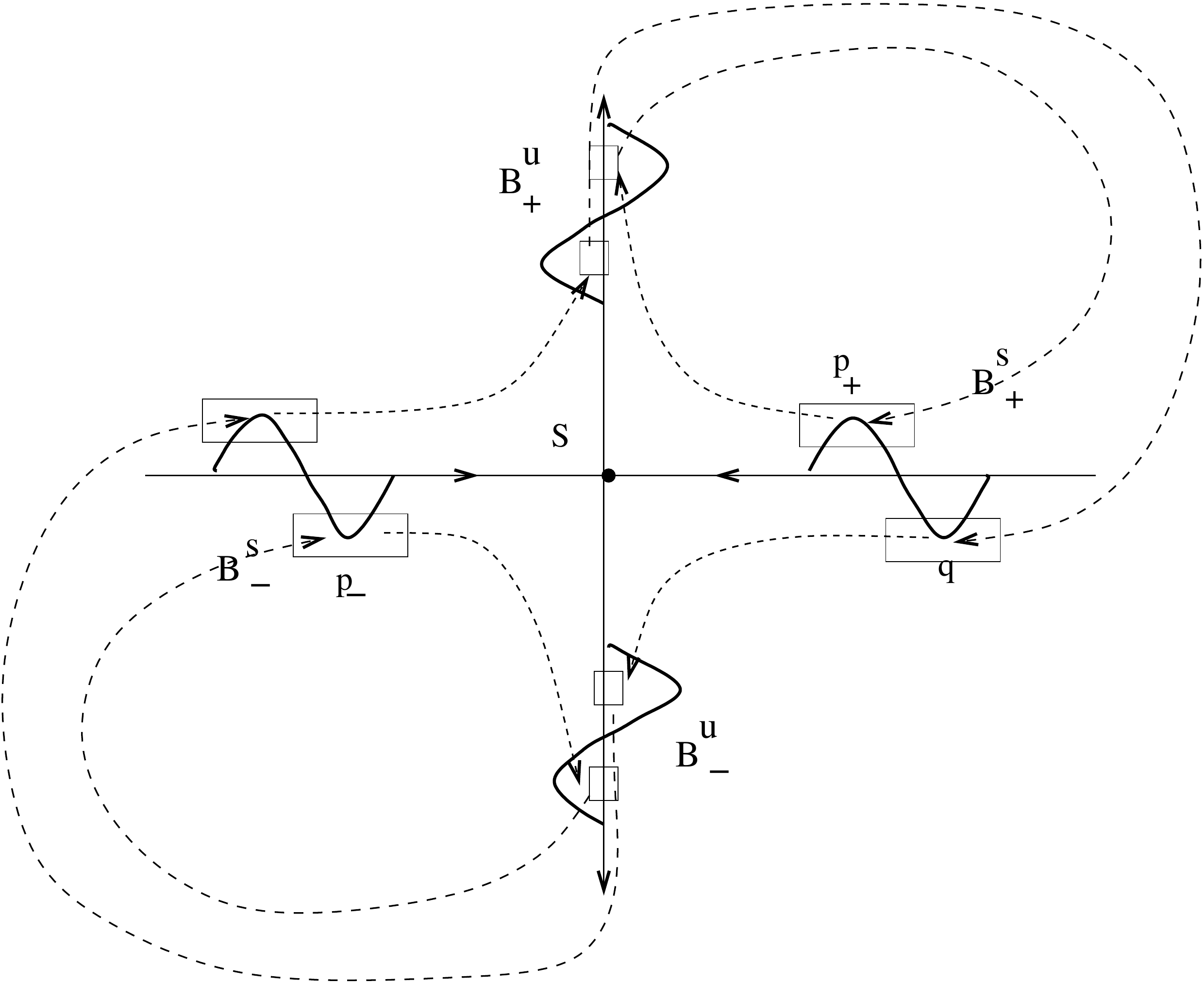}
\end{center}
\caption{Attracting periodic points of  $f_t.$}
\label{map-2}
\end{figure}

\begin{lemma}\label{localiza1} For any small $t$, if $x\in \Omega(f_t)$ then either $x\in \{R_1, R_2, R_3\}$ or belongs to $ B_{t^{\frac{1}{r}}}(\ga^+\cup \ga^-)$ (a neighborhood of radius  $t^{\frac{1}{r}}$ of $\ga^+\cup \ga^-$).  In particular, if $x\notin \{S, R_1, R_2, R_3\}$   then there exists an iterate of $x$ that belongs to $ B^s_{t^{\frac{1}{r}}}(L^s_+\cup L^s_-).$
 
\end{lemma}

\begin{proof} It follows immediately from the fact that $f_t$ restricted to the complement of  $B^u_{t^{\frac{1}{r}}}$ coincides with $f_{0}.$
 
\end{proof}

Now, using that $f_t$ is dissipative in a neighborhood of $S $ and Lemma \ref{localiza1} we conclude that the non-wandering set inside $B^s_+$ is contained in a small strip around $f^2_t(L^u_+)$; in the same way, the non-wandering set inside $B^s_-$ is contained in a small strip around $f^2_t(L^u_-).$ We fix  first $\al$ larger than $2$ and smaller than $3$.

\begin{lemma}\label{nonwander} For small $t$, there exists $\xi$ verifying $\al<\xi<3$  such that if $\La_{t}=\cap_{n\in \Z} f^n_t(W)$ then  
\begin{enumerate}
\item $\La_t\cap B^s_+\subset B_{t^\xi}(f^2_t(L^u_+)),$ and so  $\Omega(f_t)\cap B^s_+\subset B_{t^\xi}(f^2_t(L^u_+));$
 \item   $\La(f_t)\cap B^s_-\subset B_{t^\xi}(f^2_t(L^u_-)),$ and so $\Omega(f_t)\cap B^s_-\subset B_{t^\xi}(f^2_t(L^u_-));$
\end{enumerate}
where $B_{t^\xi}(f^2_t(L^u_\pm))$ denotes the neighborhood of size $\xi$ of $f^2_t(L^u_\pm).$
In particular, it follows that for any  $x\in \Omega(f_t)$ then either $x\in \{S, R_1, R_2, R_3\}$ or there exists an iterate of $x$ that belongs to $ B_{t^\xi}(f^2_t(L^u_-))\cup  B_{t^\xi}(f^2_t(L^u_+)).$

\end{lemma}

\begin{proof} Recall that $f_t$ coincides with $f_{0}$ outside a neighborhood of size of $t^{\frac{1}{r}}$ of $L^u_+\cup L^u_-$ and so from Lemma \ref{localiza1}  we have to consider the non-wandering set inside $B^s_{t^{\frac{1}{r}}}(L^u_+\cup L^u_-)$. Let $m$ be the first positive integer such that  $\s^m t^{\frac{1}{r}}\in L^u_+\cup L^u_-$  and this implies that
for any $z\in  B^s_{t^{\frac{1}{r}}}$ such that there exists a first positive integer $n_z$ with $f^{n_z}(z)\in B^u_+,$ then $dist(L(z), L^u)<\la^{n_z}$ and $n_z\geq m.$
Since $\la=\s^{-\ga}$  then $\la^{n_z}\leq  \la^m <\s^{-\ga m}\leq  t^{\frac{\ga}{r}} $ and from the election of $\ga$ (see second item in subsection \ref{initial}), which verifies that $\frac{\ga}{r}>3 $ it follows that there exists $\xi$ larger than $\al$ (and without loss of generality smaller than $3$)  such that $\frac{\ga}{r} >\xi$ and so 
 $dist(f^{n_z}_t(z), L^u)<\la^{n_z}< t^{\xi}$. From the definition of $f^2_t$ the thesis of the lemma is concluded.

\end{proof}

\begin{remark}\label{rob t} Observe that the thesis of lemmas \ref{pozo1}, \ref{pozo2}, \ref{pozo3}, and \ref{nonwander} hold for small $C^r$ perturbations of $f_t$.
 
\end{remark}

\begin{remark} In a similar way as in the proof of Lemma \ref{nonwander}  it can be concluded  that
 \begin{enumerate}
 \item   if $x\in B^s_\pm$ then either $x\in W^s_{loc}(S)$ or there exists a forward  iterate that return to  $B^s_+\cup B^s_-$;
\item  if $x\in B^u_\pm $ then either $x\in W^u_{loc}(S)$ or there exists a backward   iterate that return to  $B^u_+\cup B^u_-.$ 
\end{enumerate}
\end{remark}

In what follows we prove the existence of a dominated splitting on the non wandering set excluding  the attracting and repelling points $p_+, p_-, q,$ $R_1, R_2, R_3.$ A compact invariant  set $\La$ has a dominated splitting if there exist two complementary invariant subbundles $E\oplus F$  by the action of the derivative such that
$||Df^n_{E(z)}||||Df^{-n}_{F(f^n(z)}||< \frac{1}{2}$ for any $z\in \La$ and any positive integer  $n$ sufficiently large. The existence of dominated splitting is equivalent to the existence of invariant cone fields, i.e., a cone field $\{{\cal C}(z)\}_{z\in \La}$ such that 
$$Df^n({\cal C}(z))\subset interior ({\cal C}(f^n_t(z))$$ for any $z\in \La$ and any positive integer  $n$ sufficiently large.

\begin{lemma} \label{hyplemma} For any $t$ verifying the these44s of Lemmas \ref{pozo1}, \ref{pozo2} and \ref{pozo3}, the set $\La_{t}$ has a dominated splitting. In particular, $\Omega(f_t)
\setminus \{p_+, p_-, q, R_1, R_2, R_3\}$ has a dominated splitting. 
 \end{lemma}

\begin{proof} It is enough to show that for the set of points that is not contained in the local  basin of attraction of the periodic points given by Lemmas \ref{pozo1}, \ref{pozo2} and \ref{pozo3} it is possible to build an invariant unstable cone field. More precisely, this cone field is defined in 
\begin{eqnarray}\label{set}
B_{t^\xi}(L^u_+)\cup B_{t^\xi}(L^u_-)\bigcup B_{t^\xi}(f^2_t(L^u_+))\cup B_{t^\xi}(f^2_t(L^u_-))\bigcup W^s_{loc}(S)
 \end{eqnarray}
which is the region that contains the part of the non-wandering set inside  $B^u_+\cup B^u_-\cup B^s_+\cup B^s_-$. Latter, by standard procedure, the cone field  is extended everywhere by iteration and taking the closure. 
Therefore, to show that the cone field is invariant, is enough to check it in  the region (\ref{set}). 
For points $(x,y)\in B_{t^\xi}(L^u_+)\cup B_{t^\xi}(L^u_-)$  the cone field has the vertical vector $(0,1)$ as direction and slope $t^\xi$, i.e.,
$${\cal C}(x,y)=\{v: slope(v, (0,1))\leq t^{\xi}\}.$$
For points $(x',y')=f^2_t(x,y)\in B_{t^\xi}(f^2_t(L^u_+))\cup B_{t^\xi}(f^2_t(L^u_-))$
it is taken a cone field with direction tangents  to $f^2_t(\{x\}\times L^u_+)$ with $|x|< \xi$ fixed and slope $t^{\xi}$. More precisely, given a point  $(x', y')=(f^u(y), x+f_t^s(y))\in B_{t^\xi}(f^2_t(L^u_+))$  and defining $w_{(x',y')}=  ({f^u}'(y), {f^s_t}'(y))$ then 
$${\cal C}(x', y')=\{v: slope(v, w_{(x',y')})\leq t^{\xi}\}.$$ 
For  points that belong to the local stable manifold of the saddle $S$, we take the forward  iterate of the cone defined in $[B_{t^\xi}(f^2_t(L^u_+))\cup B_{t^\xi}(f^2_t(L^u_-))]\cap W^s_{loc}(S)$ and for the saddle fixed point $S$ we take a cone centered at its unstable direction.  

Now we proceed to prove that the cone field are invariant.  From Lemma \ref{pozo1}, Lemma \ref{nonwander} and item (4) in the definition of $f_t^s$  follows that for any point $(x', y')\in B^s_+\cap \Omega(f_t)$ then   
$$slope(w_{(x',y')}, (1, 0))> t^{\frac{\al}{2}},$$ in particular, if $v\in {\cal C}(x', y')$ then 
$$slope(w_{(x',y')}, (1, 0))> t^{\frac{\al}{2}}-t^\xi> \frac{1}{2}t^{\frac{\al}{2}}.$$
Observe that by definition, 
\begin{eqnarray}\label{cono1}
 f^2_t({\cal C}(x,y))={\cal C}(f^2_t(x, y)).
\end{eqnarray}
Let $(x,y)$ be a point in $B^s_+$ that does not belong to the local stable manifold of $S$, i.e $y> 0$. Let   $m=m(x,y)$ be the first integer such that $f^{m}_t(x,y)\in B^+_+$,   i.e. $m$ is the first integer such that $\s^m y \in B^u_+$. It  remains to show that for $(x,y)\in B^s_+$ holds $$D_{(x,y)}f^{m+2}_t({\cal C}(x,y))\subset {\cal C}(f^{m+2}(x,y)).$$
From (\ref{cono1}) it is enough  to show that for $(x,y)\in B^s_+$ holds that 
\begin{eqnarray}\label{cono2}
D_{(x,y)}f^{m}_t({\cal C}(x,y))\subset {\cal C}(f^{m}_t(x,y)).
 \end{eqnarray}
Observe  that $m$ is larger than $m_0$ such that $\s^{m_0}t> b$.  So, giving $v=(1, v_2)\in {\cal C}_{(x,y)}$ then  $Df^m_t(v)=(\la^m, v_2\s^m)=v_2\s^m(\frac{\la^m}{v_2\s^m}, 1) $ with $v_2> t^\frac{\al}{2}-t^\xi> \frac{1}{2}t^\frac{\al}{2}$ and therefore
$$slope(Df^m_t(v), (0,1))\leq   \frac{\la^m}{v_2\s^m}\leq  t^{\ga-\frac{\al}{2}-1}$$. Recalling that $\ga>6$, $\al<\xi<3$, we have $\ga-\frac{\al}{2}-1>\xi$, which implies that the slope of $Df^m_t(v)$ with $(0,1)$ is strictly  smaller that $t^\xi$ and so (\ref{cono2}) is proved. 
\end{proof}

To finish the proof of  Proposition \ref{hyp},  we use a result proved in \cite{pujals-sambarino} (see Theorem B) that states that a generic Kupka-Smale $C^2$ diffeomorphism with a dominated splitting in its non-wandering set is Axiom A.

\begin{remark} From remark \ref{rob t},  we can assume without loss of generality that $f_t$ is Kupka-Smale and moreover it does not contain normally hyperbolic invariant curves. Therefore, from Theorem B in \cite{pujals-sambarino} it follows that $\La_{t}$ is a hyperbolic set. Since there are transversal homoclinic points associated to $S$, it follows that $\La_t$ contains a non trivial hyperbolic set.
\end{remark}

\begin{proof}[Proof of Proposition \ref{hyp}] To conclude, we have to show the strong transversality condition: the stable and unstable manifold  of any basic piece given by the  spectral decomposition intersect transversally. This follows immediately using that in  $\La_t
\setminus \{p_+, p_-, q, R_1, R_2, R_3\}$ there is a dominated splitting. In fact, if $z\in W^s(x)\cap W^u(x')$ for some $x, x' \in \Omega_t
\setminus \{p_+, p_-, q, R_1, R_2, R_3\}$,  $z\in \La_t$ and $W^s_{[x,z]}(x)\cup  W^u_{[x,z']}(x')\subset \La_t$ where $W^s_{[x,z]}(x) $ is the connected arc of $W^s(x)$ that contains $x$ and is bounded by $x$ and $z$ (and $W^u_{[x,z']}(x')$ is the connected arc of $W^u(x')$ that contains $x'$  and is bounded by $x'$ and $z$).
So, $T_z W^s(x)$ is tangent to the subbundle $E_z$ and  $T_z W^u(x)$ is tangent to the subbundle $F_z$ provided by the dominated splitting, and therefore the intersection is transversal.
\end{proof}

\subsection{About the construction of $f_t$}\label{about} We consider a tubular neighborhood $\cal T$  around $\ga^+$ that contains  $B_+^u$ and $B^s_-$ for the flow $\Phi_\frac{1}{m}$. Moreover, using appropriate coordinates we can assume that ${\cal T}=[-1, m+2]\times [-1,1]$, 
$B_+^u=[0,1]\times [-1,1]$, $B_+^s=[m,m+1]\times [-1,1]$ and $$\Phi_\frac{1}{m}(x, y)= (x+1, y).$$
Now we consider a map $g: \R\to \R$ with support in $[-\e, 1+\e]$ such that 
\begin{enumerate}
 \item $g_0:=g_{/[0,1]}$ verifies that $g_0(0)=g_0(1/2)=g_0(1)=0$, only has two critical points at $1/4,$ and $3/4$, $g_0(1/4)=1,$  $g_0(3/4)=1+\de$ , $g_0$  is positive in $(0,1/2)$ and negative in $(1/2,1);$
\item $g_0'$ is increasing in $[0,1/2]$ and decreasing in $[1/2,1];$
\item $g_1:=g_{/[1,1+\e]}$ has only one critical point and $g_1(x)<<g_0(x-1).$
\end{enumerate}
Moreover, taking $\e$ small, the maps can be chosen in such a way that $\hat g=g_0(x)+g_1(x+1)$ verifies in $[0,1]$ the same properties that $g_0$ verifies. Moreover, taking $g_t=t g$ and $\hat g_t=t\hat g=g_{0t}(x)+g_{1t}(x+1)$, for $t$ small then
the $C^r$ norm of  $\hat g_t$ is also small. Now observe that if $$f_t(x,y)= (x+1, y+b(y)g_t(x)),$$ where $b$ is bump function that is zero outside $[-1,1]$ and is equal to $1$ in $[-1/2, 1/2]$  then $f_t$ is a diffeomorphism, and for $(x,y)\in B_+^u$ with $y\in [-1/2, 1/2]$ follows that  $$f^m_t(x,y)= (x+m, y+\hat g_t(x))$$ and observe that after the coordinates changes in $\cal T$ the map has the properties required for $f_t$ in $B_+^u$.

\bibliographystyle{koro}
\bibliography{shadow}
\end{document}